\documentclass[12pt]{amsart}
\usepackage{amsfonts,bbm,esint,mathrsfs,amssymb,amsmath,amsfonts,amsthm,color}

\numberwithin{equation}{section}

\newtheorem{thm}{Theorem}[section]
\newtheorem{lem}[thm]{Lemma}
\newtheorem{cor}[thm]{Corollary}
\newtheorem{pro}[thm]{Proposition}
\theoremstyle{definition}

\newtheorem{rem}[thm]{Remark}

\def \ir {\iiint_{E(\epsilon,R_1,R_2)}}
\def \irb {\int_{|x_3|\leq \frac{1}{R_1R_2}}\int_{|x_2|\leq \frac{1}{R_2}}\int_{|x_1|\leq
\frac{1}{R_1}}}

\def \K {\mathcal{K}}
\def \Dx {\Delta_{x_1,h_1}}
\def \Dy {\Delta_{x_2,h_2}}
\def \Dz {\Delta_{x_3,h_3}}
\def \ta {{\theta_1}}
\def \tb {{\theta_2}}
\def \12 {{\frac{1}{2}}}
\def \ha {{h_1}}
\def \hb {{h_2}}
\def \hc {{h_3}}
\def \taa {{\alpha\theta_1}}
\def \tab {{\beta\theta_1}}
\def \tac {{\gamma\theta_1}}
\def \xr {\delta_1 \leq |x_1| \leq r_1}
\def \yr {\delta_2 \leq |x_2| \leq r_2}
\def \zr {\delta_3 \leq |x_3| \leq r_3}
\def \xen {\int_{\epsilon_1 \leq |x_1| \leq N_1}}
\def \yen {\int_{\epsilon_2 \leq |x_2| \leq N_2}}
\def \zen {\int_{\epsilon_4 \leq |x_3| \leq N_4}}
\def \x8n {\int_{8 \leq |x_1| \leq N_1}}
\def \y8n {\int_{8 \leq |x_2| \leq N_2}}
\def \z8n {\int_{8 \leq |x_3| \leq N_4}}
\def \xe8 {\int_{\epsilon_1 \leq |x_1| \leq 8}}
\def \ye8 {\int_{\epsilon_2 \leq |x_2| \leq 8}}
\def \ze8 {\int_{\epsilon_4 \leq |x_3| \leq 8}}
\def \Kxi {\frac{1}{\xi}\mathcal{K}(x_1,x_2,\frac{x_3}{\xi})}
\def \Kb {\widetilde{\mathcal{K}}(x_1,x_2,x_3)}
\def \xyzall  {\Big(\Big|\frac{x_1x_2\xi}{x_3\chi\eta}\Big|+\Big|\frac{x_3\chi\eta}{x_1x_2\xi}\Big|\Big)^{-\tb}}

\allowdisplaybreaks

\begin{document}

\baselineskip 17.2pt \hfuzz=6pt

\title[]
{A class of singular integrals\\ associated with Zygmund dilations}

\bigskip

\author[]
{Yongsheng Han, Ji Li, Chin-Cheng Lin, and Chaoqiang Tan}

\address{Department of Mathematics\\
Auburn University\\
Auburn, AL 36849, USA}
\email{hanyong@mail.auburn.edu}

\address{Department of Mathematics\\
Macquarie University\\
NSW, 2109, Australia}
\email{ji.li@mq.edu.au}

\address{Department of Mathematics\\
National Central University\\
Chung-Li 320, Taiwan}
\email{clin@math.ncu.edu.tw}

\address{Department of Mathematics\\
Shantou University\\
Shantou, Guangdong 515063, China}
\email{cqtan@stu.edu.cn}

\thanks{}


\begin{abstract}
          The main purpose of this paper is to study multi-parameter singular integral operators which commute with Zygmund dilations. We introduce a class of singular integral operators associated with Zygmund dilations and show the boundedness for these
          operators on $L^p, 1<p<\infty$, which covers those studied by Ricci--Stein \cite{RS} and Nagel--Wainger \cite{NW}.
\end{abstract}

\maketitle

%\tableofcontents

%%%%%%%%%%%%%%%%%%%%%%%%%%%%%%%%%%%%%%%%%%%%%%%%%%%%%%%%%%%%%%%%%%%%%%%%%%%%%%%%%%%%%%%%%
%%%%%%%%%%%%%%%%%%%%%%%%%%%%%%%%%%%%%%%%%%%%%%%%%%%%%%%%%%%%%%%%%%%%%%%%%%%%%%%%%%%%%%%%%
\section{Introduction}
Ricci--Stein \cite{RS} introduced multi-parameter singular integral operators
and Fefferman--Pipher \cite{FP} considered specific singular integral operators associated with Zygmund dilations.
The boundedness for these operators on $L^p$ and weighted $L^p_w,1<p<\infty,$ was obtained by Ricci--Stein \cite{RS}
and Fefferman--Pipher \cite{FP}, respectively. The main purpose of this paper is to introduce a class of singular integral operators associated with Zygmund dilations and show the boundedness for these
operators on $L^p, 1<p<\infty$, which cover those studied by Ricci--Stein \cite{RS} and Nagel--Wainger \cite{NW}.

We now set our work in context. In their well-known theory, Calder\'on and Zygmund \cite{CZ}
introduced certain convolution singular integral operators on $\Bbb
R^n$ which generalize the Hilbert transform on $\Bbb R^1.$ They
proved that if $T(f)=\mathcal{K}*f,$ where $\mathcal{K}$ is defined
on $\Bbb R^n$ and satisfies the analogous estimates as $\frac{1}{x}$
does on $\Bbb R^1,$ namely
\begin{equation*}
|\mathcal{K}(x)|\le \frac{C}{|x|^n},
\end{equation*}
\begin{equation*}
|\nabla \mathcal{K}(x)|\le \frac{C}{|x|^{n+1}},
\end{equation*}
and
\begin{equation*}
\int_{a<|x|<b}\mathcal{K}(x)dx=0\qquad\text{for all}\ \
0<a<b,
\end{equation*}
then $T$ is bounded on $L^p(\Bbb R^n)$ for $1<p<\infty.$

The core of this theory is that the regularity and cancellation
conditions are invariant\break with respect to the
one-parameter family of dilations on $\Bbb R^n$ defined by
$\delta(x_1, x_2, \cdots, x_n)=\break (\delta x_1, \cdots, \delta
x_n)$, $\delta>0,$ in the sense that the kernel $\delta^n\mathcal
K(\delta x)$ satisfies the same conditions with the same bound as
$\mathcal K(x)$. Indeed, the classical singular integrals, maximal functions
and multipliers are invariant with respect to such one-parameter
dilations. The one-parameter theory is well understood up to now.
On the other hand, the multiparameter theory of $\Bbb R^n$ began
with Zygmund's study of the strong maximal function, which is
defined by
$$\mathcal M_n(f)(x)=\sup_{R\ni x}{{1}\over{|R|}} \int _R |f(y)| dy,$$
where $R$ are the rectangles in $\Bbb
R^n$ with sides parallel to the axes, and then continued with
Marcinkiewicz's proof of his multiplier theorem.
If we consider the
family of product dilations defined by $\delta(x_1, x_2, \cdots,
x_n)=(\delta_1 x_1, \cdots, \delta_n x_n), \delta_i>0, i=1,..., n,$
then the strong maximal function and Marcinkiewicz's multiplier are
invariant under the product dilations. The multiparameter dilations
are also associated with problems in the theory of differentiation
of integrals. Jensen--Marcinkiewicz--Zygmund \cite{JMZ} proved that
the strong maximal function in $\Bbb R^n$ is bounded from the Orlicz
space $L(1+(\log ^+L)^{n-1})$ to weak $L^1$. Zygmund further
conjectured  that if the rectangles in $\Bbb R^n$ had $n$ side
lengths which involve only $k$ independent variables, then the
resulting maximal operator should behave like $\mathcal{M}_k,$ the
$k$-parameter strong maximal operator. More precisely, for $1 \le k
\le n$, and for positive functions $\phi_1, \cdots, \phi_n$ as the
side-lengths of the given collection of rectangles where the maximal
function is defined, each one depending on parameters $t_1 > 0, t_2>
0, \cdots, t_k > 0$, assuming arbitrarily small values and
increasing in each variable separately, then the resulting maximal
function would be bounded from $L(1+(\log ^+L)^{k-1})$ to weak $L^1$
according to Zygmund's conjecture.
Cordoba \cite{Co} showed that for the unit cube $Q$ in $\Bbb R^3,$
$$\vert \lbrace x\in Q: {\mathcal M}_{\frak z}f(x)>\lambda \rbrace
\vert \leq {{C}\over{\lambda }}\Vert  f\Vert_{L\log L(Q)},$$
where ${\mathcal M}_{\frak z}f$ denotes the maximal function on $\Bbb R^3$
defined by
$${\mathcal M}_{\frak z} f (x)
 = \sup_{R\ni x}\frac{1}{|R|}\int_{R}|f(u)|du.$$
The supremum above is taken over all rectangles with sides parallel
to the axes and side lengths of the form $s, t$, and $\phi(s, t).$
Cordoba's result was generalized to the case of $\phi_1(s, t)$,
$\phi_2(s,t)$, $\phi_3(s, t)$ by Soria \cite{So} with some assumptions
on $\phi_1, \phi_2,\phi_3$. Moreover, Soria showed that
Zygmund's conjecture is not true even when $\phi_1(s, t)=s$,
$\phi_2(s, t)=s\phi(t)$, $\phi_3(s,t)=s\psi(t)$,
with $\phi, \psi$ being positive and increasing functions.

In \cite{RFS} Fefferman and Stein generalized the singular integral
operator theory to the product space. They took the space $\Bbb
R^n\times \Bbb R^m$ along with the two-parameter family of dilations
$(x, y)\mapsto (\delta_1 x, \delta_2 y), (x,y)\in \Bbb R^n\times
\Bbb R^m, \delta_1,\delta_2>0.$ Those operators considered in \cite{RFS}
generalize the double Hilbert transform on $\Bbb R^2$ given by
$H(f)=f\ast\frac{1}{xy}$ and are of the form $T(f)=\mathcal{K}\ast
f,$ where the kernel $\mathcal{K}$ is characterized by the
cancellation properties
\begin{equation}\label{eq 1.4}
\int_{a<|x|<b}\mathcal{K}(x,y)dx=0\qquad\text{for all}\ \ 0<a<b\ \text{and}\ y \in \Bbb R^m,
\end{equation}
\begin{equation}\label{eq 1.5}
\int_{a<|y|<b}\mathcal{K}(x, y)dy =0\qquad\text{for all}\ \ 0<a<b\ \text{and}\ x \in \Bbb R^n,
\end{equation}
and the regularity conditions
\begin{equation}\label{eq 1.6}
|\partial^\alpha_{x} \partial^\beta_{y} \mathcal{K}(x,y)| \le C_{\alpha,
\beta}|x|^{-n-|\alpha|}|y|^{-m-|\beta|}.
\end{equation}

Under the conditions (\ref{eq 1.4}) -- (\ref{eq 1.6}), Fefferman and Stein proved the
$L^p, 1<p<\infty,$ boundedness of the product convolution operators
$T(f)=\mathcal K\ast f.$ See \cite{RFS} for more details. Note that
the kernel $\mathcal K$ satisfying the conditions (\ref{eq 1.4}) -- (\ref{eq 1.6}) is
invariant with respect to the product dilation in the sense that the
kernel $\delta_1^n\delta_2^m\mathcal K(\delta_1 x,\delta_2 y)$
satisfies conditions (\ref{eq 1.4}) -- (\ref{eq 1.6}) with the same bound. For more discussions about the multiparameter product
theory, see \cite{Carleson,Chang,CF3,CF2,CF1,CoF,CFS,Fe4,Fe1,Fe2,Fe3,GS,J1,J2,J3,P,St} and in particular  the survey article of R. Fefferman \cite{Fe3} for development in this area. For singular integrals with flag kernels, see \cite{MRS,NRS,NRSW}.

It has been widely considered that the next simplest
multiparameter group of dilations after the product multiparameter
dilations is the so-called the Zygmund dilation defined on $\Bbb R^3$ by
$\rho_{s,t}(x_1,x_2,x_3) = (sx_1, tx_2, stx_3)$ for $s, t > 0$. Indeed, as far
as ${\mathcal M}_{\frak z}$ is concerned, E. M. Stein was the
first to link the properties of maximal operators associated with Zygmund dilations to
boundary value problems for Poisson integrals on symmetric spaces,
such as Siegel's upper half space. See the survey paper of R. Fefferman \cite{Fe1}
on the future direction of research of multiparameter analysis on Zygmund dilations.

There are two operators intimately associated with Zygmund dilations.
One is the maximal operator ${\mathcal M}_{\frak z}$ as mentioned
above. Another is the singular integral operator $T_{\frak z}$ introduced by Ricci and Stein
\cite{RS}, which commutes with this dilation. A special class of singular integral operators
$T_{\frak z}$ considered by Ricci and Stein is of the form
$T_{\frak z}f=f* \mathcal{K}$, where
$$
\mathcal{K}(x_1,x_2,x_3) =\sum _{k,j\in {\Bbb Z}}2^{-2(k+j)}\phi ^{k,j}
  \Big({{x_1}\over{2^k}},{{x_2}\over{2^j}},{{x_3}\over{2^{k+j}}}\Big)
$$
and the functions $\phi ^{k,j}$ are supported in an unit cube in
$\Bbb R^3$ with a certain amount of uniform smoothness and satisfy cancellation conditions
\begin{equation*}
\int _{{\Bbb R}^2}\phi ^{k,j}(x_1,x_2,x_3)dx_1 dx_2
=\int _{{\Bbb R}^2} \phi^{k,j}(x_1,x_2,x_3)dx_2 dx_3=\int _{{\Bbb R}^2}\phi ^{k,j}(x_1,x_2,x_3)dx_3 dx_1=0.
\end{equation*}
It was shown in \cite{RS} that $T_{\frak z}$ is
bounded on $L^p(\Bbb R^3)$ for all $1<p< \infty.$ Particularly, as mentioned in \cite{FP}, the above cancellation
conditions are also necessary for the boundedness of the above mentioned operators on $L^2(\Bbb R^3).$ It is easy to
see that if the dyadic Zygmund dilation is given by
$(\delta _{{2^j},{2^k}}f)(x_1,x_2,x_3)=2^{2(j+k)}
f(2^j x_1,2^k x_2, 2^{(j+k)} x_3),$ then we obtain that  $(\delta _{{2^j},{2^k}}T_{\frak z}(f))(x_1,x_2,x_3)=
T_{\frak z}(\delta _{{2^j},{2^k}}f)(x_1,x_2,x_3).$ This
means that the operators studied by Ricci and Stein commute with
Zygmund dilations of dyadic form. See \cite{RS} for more details.

R. Fefferman and Pipher \cite{FP} further showed that $T_{\frak z}$
is bounded in $L^p_w$ spaces for $1<p<\infty$ when the weights
$w$'s satisfy an analogous condition of Muckenhoupt associated with
Zygmund dilations. Related to the theory of operators like
${\mathcal M}_{\frak z}$ and $T_{\frak z}$, several authors
have considered singular integrals along surfaces in
$\Bbb R^n.$ See, for example, Nagel--Wainger \cite{NW}.

To achieve our goal, the first aim of this paper is to develop a class of singular
integral operators associated with Zygmund dilations, which covers
those introduced in \cite{RS} and prove the $L^2$ boundedness. The second aim is to show the $L^p$, $1<p<\infty$, boundedness for this class of singular integral operators.

Suppose that $\mathcal{K}(x_1,x_2,x_3)$ is a
function defined on $\Bbb R^3$ away from the
union $\lbrace 0, x_2,x_3\rbrace\cup\lbrace x_1, 0,x_3\rbrace\cup\lbrace x_1, x_2,0\rbrace$ and all $\alpha, \beta$ and $\gamma$ are integers taking only values 0 and 1. We define
\begin{eqnarray*}
\Dx^\alpha \K(x_1,x_2,x_3)=\alpha\K (x_1+h_1,x_2,x_3)- \K (x_1,x_2,x_3),\qquad \alpha=0
\text{ or } 1,
\end{eqnarray*}
\begin{eqnarray*}
\Dy^\beta \K(x_1,x_2,x_3)=\beta\K (x_1,x_2+h_2,x_3)- \K (x_1,x_2,x_3),\qquad \beta=0 \text{
or } 1,
\end{eqnarray*}
and
\begin{eqnarray*}
\Dz^\gamma \K(x_1,x_2,x_3)=\gamma\K (x_1,x_2,x_3+h_3)-\K (x_1,x_2,x_3),\qquad \gamma=0
\text{ or } 1.
\end{eqnarray*}
For simplicity, we denote $\Dx=\Dx^1,$ $\Dy =\Dy^1$ and $\Dz
=\Dz^1$.

The ``regularity'' conditions considered in this paper are
characterized by
\begin{equation}\tag{R}
|\Dx^\alpha\Dy^\beta\Dz^\gamma \mathcal{K}(x_1,x_2,x_3)| \leq
\frac{C|\ha|^{\taa} |\hb|^{\tab} |\hc|^{\tac}}{|x_1|^{\taa+1}
|x_2|^{\tab+1} |x_3|^{\tac+1} \big(|\frac{x_1x_2}{x_3}|+|\frac{x_3}{x_1x_2}|
\big)^{\tb}}
\end{equation}
for all $0 \leq \alpha \leq 1, 0 \leq \beta+\gamma\leq 1$ or $0 \leq
\alpha+\gamma \leq 1, 0 \leq \beta\leq 1,$ and $|x_1|\geq 2 |\ha|>0, |x_2|\geq
2 |\hb|>0, |x_3|\geq 2 |\hc|>0$, $\ha,\hb,\hc\in \Bbb R$ and some
$0<\ta\leq 1,0<\tb<1.$

Note that for any fixed non zero two variables, say, $x_1\not=0$ and
$x_2\not=0,$ $\mathcal{K}(x_1,x_2,x_3)$ is an integrable function with
respect to the variable $x_3$ and the resulting integral $\widetilde
K(x_1,x_2)=\int_{\Bbb R}\mathcal K(x_1,x_2,x_3)dx_3,$ as a kernel on $\Bbb R^2,$
satisfies the regularity conditions of the classical product kernel
on $\Bbb R^2$ as studied by Fefferman and Stein in \cite{RFS}. These
facts, as mentioned above, can also be easily checked for singular integral operators
studied by Ricci and Stein.

In this paper, we will consider three kinds of cancellation conditions. The first one
is given by
\begin{equation} \tag{C1.a}
\bigg|\int_{\xr }\int_{\yr }\int_{\zr } \K (x_1,x_2,x_3)dx_1dx_2dx_3 \bigg|
\le C
\end{equation}
uniformly for all $\delta_1,\delta_2,\delta_3,r_1,r_2,r_3>0;$
\begin{equation} \tag{C1.b}
\bigg|\int_{\xr }\int_{\yr } \Dz^\gamma \K (x_1,x_2,x_3)dx_1dx_2 \bigg| \le
\frac{C|\hc|^\tac}{|x_3|^{\gamma\ta+1}}
\end{equation}
uniformly for all $\delta_1,\delta_2,r_1,r_2>0,$ $|x_3|\geq 2|\hc|>0$
and $0 \le \gamma\leq 1;$
\begin{equation} \tag{C1.c}
\bigg|\int_{\yr }\int_{\zr } \Dx^\alpha \mathcal{K}(x_1,x_2,x_3)dx_2
dx_3 \bigg| \leq \frac{C|\ha|^{\taa}}{|x_1|^{\taa+1}}
\end{equation}
uniformly for all $\delta_2,\delta_3,r_2,r_3>0,$ $|x_1| \geq 2|\ha|>0$
and $0 \le \alpha\leq 1;$
\begin{equation} \tag{C1.d}
\bigg|\int_{\xr }\int_{\zr } \Dy^\beta \mathcal{K}(x_1, x_2,
x_3)dx_3dx_1\bigg| \leq \frac{C|\hb|^\tab}{|x_2|^{\tab+1}}
\end{equation}
uniformly for all $\delta_1,\delta_3,r_1,r_3>0,$ $|x_2| \geq 2|\hb|>0$
and $0 \le \beta\leq 1.$

The regularity conditions (R) and the cancellation conditions (C1.a)
- (C1.d) imply the following $L^2$ boundedness.

%, namely one of main
%results of this paper.

\begin{thm}\label{Thm 1.1}
Suppose that $\mathcal{K}$ is a function defined on $\Bbb R^3$ and
satisfies the conditions {\rm(R)} and {\rm(C1.a)} -- {\rm(C1.d)}. Set $\mathcal
K_{\epsilon}^N(x_1,x_2,x_3)=\mathcal K(x_1,x_2,x_3)$ if $\epsilon_1 \leq |x_1|
\leq N_1, \epsilon_2 \leq |x_2| \leq N_2$ and $\epsilon_3\leq |x_3| \leq
N_3$ and $\mathcal K_{\epsilon}^N(x_1,x_2,x_3)=0$ otherwise, where
$\epsilon=(\epsilon_1,\epsilon_2,\epsilon_3)$ and $N=(N_1,N_2,N_3)$
for all $0<\epsilon_1\leq N_1<\infty,\epsilon_2\leq N_2<\infty,$ and
$\epsilon_3\leq N_3<\infty.$ Then, the operator $\mathcal
K_{\epsilon}^N\ast f$ is bounded on $L^2(\Bbb R^3)$ and moreover,
$$\|\mathcal K_{\epsilon}^N*f\|_{L^2(\Bbb R^3)}\leq A\|f\|_{L^2(\Bbb R^3)}$$
where the constant $A$ depends only on the constant $C$ but not on
$\epsilon_1,\epsilon_2, \epsilon_3, N_1,N_2$ and $N_3.$
\end{thm}

From Theorem \ref{Thm 1.1} we will deduce the existence of the
corresponding singular integrals in the $L^2$ norm as a limit of the
truncated integrals.

\begin{cor}\label{Cor 1.2}
Suppose that $\mathcal K$ is a function defined on $\Bbb R^3$ and
satisfies the conditions {\rm(R)}, {\rm(C1.a)} -- {\rm(C1.d)} and, in addition, the
three integrals
\begin{align*}
&\int_{|x_3|\leq 1}\int_{|x_2|\leq 1}\int_{|x_1|\leq 1} K_{\epsilon}^N(x_1,x_2,x_3)dx_1dx_2dx_3,\\
&\int_{|x_3|\leq 1}\int_{|x_2|\leq 1}\mathcal K_{\epsilon}^N(x_1,x_2,x_3)dx_2dx_3,\\
&\int_{|x_3|\leq 1}\int_{|x_1|\leq 1}\mathcal K_{\epsilon}^N(x_1,x_2,x_3)dx_1dx_3
\end{align*}
converge almost everywhere as $\epsilon_1,\epsilon_2,\epsilon_3\to 0$ and
$N_1,N_2,N_3\to \infty.$ Then the limit\break
$\lim_{\epsilon_1,\epsilon_2,\epsilon_3\to 0 \atop N_1,N_2,N_3\to
\infty} \mathcal K_{\epsilon}^N\ast f=\mathcal K\ast f$ exists in
the $L^2(\Bbb R^3)$ norm. Moreover,
$$\|\mathcal{K}*f\|_{L^2(\Bbb R^3)}\leq A\|f\|_{L^2(\Bbb R^3)}$$
with the constant $A$ depending only on the constant $C$.
\end{cor}

We remark in advance that the proof of Corollary \ref{Cor 1.2}
indeed implies that\break $\lim_{\epsilon_1,\epsilon_2,\epsilon_3\to 0
\atop N_1,N_2,N_3\to \infty} \mathcal K_{\epsilon}^N\ast f$ exists
in the $L^p, 1<p<\infty,$ norm for smooth functions $f$ having
compact support. This fact leads to the study of the $L^p, p\not=2,$
boundedness of the operator $\mathcal K\ast f.$ For this purpose, we
need the second kind of the cancellation conditions which are
somewhat stronger than the first ones. They are given by
\begin{equation} \tag{C2.a}
\bigg|\int_{\zr }\int_{\yr }\int_{\xr } \K (x_1,x_2,x_3)dx_1dx_2dx_3 \bigg|
\le C
\end{equation}
uniformly for all $\delta_1,\delta_2,\delta_3,r_1,r_2,r_3>0;$
\begin{align}\tag{C2.b}
\begin{split}
&\bigg|\int_{\xr } \Dy^\beta\Dz^\gamma \mathcal{K}(x_1,x_2,x_3)dx_1 \bigg|\\
&\quad\leq \frac{C|\hb|^\tab |\hc|^\tac}{|x_2|^{\tab+1} |x_3|
^{\tac+1}}\bigg(\frac{1}{\big(|\frac{r_1x_2}{x_3}|+|\frac{x_3}{r_1x_2}|
\big)^{\tb}}+\frac{1}{\big(|\frac{\delta_1 x_2}{x_3}|+|\frac{x_3}{\delta_1
x_2}| \big)^{\tb}}\bigg)
\end{split}
\end{align}
for all $ \delta_1, r_1>0, 0 \le \beta+\gamma\leq 1,$ $|x_2|\geq
2|\hb|>0 ,|z| \geq 2|\hc|>0;$
\begin{equation} \tag{C2.c}
\bigg|\int_{\zr }\int_{\yr }\Dx^\alpha \mathcal{K}(x_1,x_2,x_3)dx_2 dx_3
\bigg| \leq \frac{C|h_1|^\taa}{|x_1|^{\taa+1}}
\end{equation}
uniformly for all $\delta_2,\delta_3,r_2,r_3>0,$ $|x_1| \geq 2|\ha|>0$
and $0 \le \alpha\leq 1.$ Or
\begin{equation} \tag{C$2^\prime.$a}
\bigg|\int_{\zr }\int_{\yr }\int_{\xr } \K (x_1,x_2,x_3)dx_1 dx_2 dx_3 \bigg|
\le C
\end{equation}
uniformly for all $\delta_1,\delta_2,\delta_3,r_1,r_2,r_3>0;$
\begin{align}\tag{C$2^\prime.$b}
\begin{split}
&\bigg|\int_{\yr }\Dx ^\alpha \Dz^\gamma \K (x_1,x_2,x_3)dx_2 \bigg| \\
&\quad\le
\frac{C|h_1|^\taa |h_3|^\tac}{|x_1|^{\taa+1} |x_3|^{\tac}}
\bigg(\frac{1}{\big(|\frac{r_2x_1}{x_3}|+|\frac{x_3}{r_2x_1}|
\big)^{\tb}}+\frac{1}{\big(|\frac{\delta_2 x_1}{x_3}|+|\frac{x_3}{\delta_2
x_1}| \big)^{\tb}}\bigg)
\end{split}
\end{align}
for all $\delta_2, r_2>0, 0 \le \alpha+\gamma\leq 1$, $|x_1| \geq
2|h_1|>0$ and $|x_3| \geq 2|h_3|>0;$
\begin{equation}  \tag{C$2^\prime.$c}
\bigg|\int_{\zr }\int_{\xr }\Dy^\beta \K(x_1,x_2,x_3)dx_1dx_3\bigg| \le
\frac{C|h_2|^\tab}{|x_2|^{\tab+1}}
\end{equation}
uniformly for all $\delta_1,\delta_3,r_1,r_3>0,$ $|x_2|\geq 2|h_2|>0$
and $0 \le \beta\le 1.$

We would like to point out that the condition (C2.b) implies (C1.b)
and (C1.d) while (C$2^\prime.$b) implies (C1.b) and (C1.c), and all
the above regularity and cancellation conditions are invariant with
respect to the Zygmund dilation in the sense that the kernel
$\delta^2_1\delta^2_2\mathcal{K}(\delta_1 x_1, \delta_2 x_2,
\delta_1\delta_2 x_3)$ satisfies the same conditions with the exactly
same bounds as $\mathcal{K}(x_1,x_2,x_3)$.

The $L^p$ estimate then is given by the following

\begin{thm}\label{Thm 1.3}
Suppose that $\mathcal K$ is a function defined on $\Bbb R^3$ and
satisfies the conditions {\rm(R)} and {\rm(C2.a)} -- {\rm(C2.c)} {\rm(}or {\rm(R)},
{\rm(C$2^\prime.$a)} -- {\rm(C$2^\prime$.c))} and in addition the three integrals
\begin{align*}
&\int_{|x_3|\leq 1}\int_{|x_2|\leq 1}\int_{|x_1|\leq 1}K_{\epsilon}^N(x_1,x_2,x_3)dx_1dx_2dx_3,\\
&\int_{|x_3|\leq 1}\int_{|x_2|\leq 1}\mathcal K_{\epsilon}^N(x_1,x_2,x_3)dx_2dx_3,\\
&\int_{|x_3|\leq 1}\int_{|x_1|\leq 1} \mathcal K_{\epsilon}^N(x_1,x_2,x_3)dx_1dx_3
\end{align*}
converge almost everywhere as $\epsilon_1,\epsilon_2,\epsilon_3\to
0$ and $N_1,N_2,N_3\to \infty.$ Then the operator
$$\mathcal K\ast f:=\lim_{\epsilon_1,\epsilon_2,\epsilon_3\to 0 \atop
N_1,N_2,N_3\to \infty} \mathcal K_{\epsilon}^N\ast f$$ defined
initially on $L^2\cap L^p,1<p<\infty,$ extends to a bounded operator
on $L^p(\Bbb R^3);$ moreover,
$$\|\mathcal{K}*f\|_{L^p(\Bbb R^3)}\leq A\|f\|_{L^p(\Bbb R^3)}$$
with the constant $A$ depending only on the constant $C$.
\end{thm}

In many applications, singular integral operators are of the form
$\mathcal K\ast f$ where $\mathcal K$ is a distribution that equals a
function $\K$ on $\Bbb R^3$ away from the
union $\lbrace 0, x_2,x_3\rbrace\cup\lbrace x_1, 0,x_3\rbrace\cup\lbrace x_1, x_2,0\rbrace $ and satisfy
certain regularity and cancellation conditions. For this purpose, we
begin with recalling the bump functions introduced by Stein in \cite
{St}. A normalized bump function (n.b.f.) is a smooth function
$\phi$ supported on the unit ball and is bounded by a fixed constant
together with its gradient. The third kind of the cancellation
conditions considered in this paper is characterized by
\begin{equation} \tag{C3.a}
\bigg|\iiint \mathcal{K}(x_1,x_2,x_3)\phi(R_1 x_1, R_2 x_2, R_1R_2 x_3)dx_1dx_2dx_3
\bigg| \le C
\end{equation}
for every n.b.f. $\phi$ on $\Bbb R^3$ and all $R_1,R_2>0;$
\begin{equation}\tag{C3.b}
\bigg|\int \Dy^\beta\Dz^\gamma \K (x_1,x_2,x_3)\phi(R x_1)dx_1\bigg| \leq
\frac{C|h_2|^\tab |h_3|^\tac}{|x_2|^{\tab+1} |x_3|^{\tac+1}
\big(\big|\frac{R x_3}{x_2}\big|+\big|\frac{x_2}{R x_3}\big| \big)^{\tb}}
\end{equation}
for all $0 \le \beta+\gamma\leq 1$ , every n.b.f. $\phi$ on $\Bbb
R,$ $|x_2| \geq 2|h_2|>0,$ $|x_3| \geq 2|h_3|>0$ and all $R>0;$
\begin{equation} \tag{C3.c}
\bigg|\iint\Dx^\alpha \K (x_1,x_2,x_3)\phi(R_1 x_2, R_2x_3)dx_2dx_3\bigg| \leq
\frac{C|h_1|^\taa}{|x_1|^{\taa+1}}
\end{equation}
for all $0\leq \alpha \leq 1$, every n.b.f. $\phi$ on $\Bbb R^2$,
$|x_1| \geq 2|h_1|>0$ and all $R_1,R_2>0.$ Or
\begin{equation} \tag{C$3^\prime.$a}
\bigg|\iiint \mathcal{K}(x_1,x_2,x_3)\phi(R_1 x_1, R_2 x_2, R_1R_2 x_3)dx_1dx_2dx_3
\bigg| \le C
\end{equation}
for every n.b.f. $\phi$ on $\Bbb R^3$ and all $R_1,R_2>0;$
\begin{equation}\tag{C$3^\prime.$b}
\bigg|\int \Dx^\alpha\Dz^\gamma \K (x_1,x_2,x_3)\phi(R x_2)dx_2 \bigg| \leq
\frac{C|h_1|^\taa|h_3|^\tac}{|x_1|^{\taa+1} |x_3|^{\tac+1}
\big(\big|\frac{R x_3}{x_1}\big|+\big|\frac{x_1}{R x_3} \big|\big)^{\tb}}
\end{equation}
for all $0 \le \alpha+\gamma\leq 1$, every n.b.f. $\phi$ on $\Bbb
R$, $|x_1| \geq 2|h_1|>0, |x_3| \geq 2|h_3|>0$ and all $R>0;$
\begin{equation}  \tag{C$3^\prime.$c}
\bigg|\iint\Dy^\beta \mathcal{K}(x_1,x_2,x_3)\phi(R_1 x_1, R_2 x_3)dx_1 dx_3
\bigg| \leq \frac{C|h_2|^\tab}{|x_2|^{\tab+1}}
\end{equation}
for all $0\leq \beta \leq 1$, every n.b.f. $\phi$ on $\Bbb R^2$,
$|x_2| \geq 2|h_2|>0$ and all $R_1,R_2>0.$

\begin{thm}\label{Thm 1.4}
{\rm (a)} Suppose that $\mathcal{K}$ is a distribution  that equals a
function on $\Bbb R^3$ away from the
union $\lbrace 0, x_2,x_3\rbrace\cup\lbrace x_1, 0,x_3\rbrace\cup\lbrace x_1, x_2,0\rbrace $  and satisfies the
conditions {\rm(R)} and {\rm(C3.a)} -- {\rm(C3.c)} {\rm(}or {\rm(R)}, {\rm(C$3^\prime.$a)} --
{\rm(C$3^\prime.$c))}. Then, the operator $\mathcal K\ast f$ is bounded
on $L^p(\Bbb R^3),1<p<\infty;$ moreover,
$$\|\mathcal{K}*f\|_{L^p(\Bbb R^3)}\leq A\|f\|_{L^p(\Bbb R^3)}$$
with the constant $A$ depending only on the constant $C$.\\
{\rm (b)} Suppose that $\mathcal{K}$ is a distribution that equals a
function on $\Bbb R^3$ away from the union $\lbrace 0, x_2,x_3\rbrace\cup
\lbrace x_1, 0,x_3\rbrace\cup\lbrace x_1, x_2,0\rbrace $  and satisfies the
conditions {\rm(R)} and {\rm(C2.a)} -- {\rm(C2.c)} {\rm(}or {\rm(R)}, {\rm(C$2^\prime.$a)} --
{\rm(C$2^\prime.$c))}. Then, the operator $\mathcal K\ast f$ is bounded
on $L^p(\Bbb R^3),1<p<\infty,$ and,
$$\|\mathcal{K}*f\|_{L^p(\Bbb R^3)}\leq A\|f\|_{L^p(\Bbb R^3)}$$
with the constant $A$ depending only on the constant $C$.
\end{thm}

\begin{rem}
We would like to point out that all regularity and cancellation conditions given above are invariant with respect to Zygmund dilations. Moreover, the operators studied by Ricci and Stein, as mentioned before,
satisfy all above regularity and cancellation conditions. So
our results provide another proof of the boundedness for operators
in \cite{RS} on $L^p, 1<p<\infty.$ And the boundedness results
in this paper can be extended to higher dimensions. The consideration of regularity and cancellation
conditions in this paper leads naturally to the study of
non-convolution singular integral operators which are associated
with Zygmund dilations. We will discuss all these topics in the
forthcoming works.
\end{rem}

In the next section, we will show the $L^2$ boundedness for singular integral
operators associated with Zygmund dilations, namely Theorem \ref{Thm 1.1} and Corollary \ref{Cor 1.2}.
The proofs of $L^p$ boundedness, Theorems \ref{Thm 1.3} and \ref{Thm 1.4}, will be given in section 3.
In the last section, some
examples and applications of singular integral operators in our
class will be discussed. In particular, we show that the kernels of
singular integral operators $T_{\frak z}$ in the special class
studied by Ricci and Stein satisfy the conditions (R) and (C2.a) --
(C2.c) (or (R), (C$2^\prime.$a) -- (C$2^\prime.$c)), and (R), (C3.a)
-- (C3.c) (or (R), (C$3^\prime.$a) -- (C$3^\prime.$c)). We also show
that the operator considered by Nagel and Wainger \cite{NW}, where
only the $L^2$ boundedness is proved, belongs to our class, and
therefore, as a consequence of Theorem \ref{Thm 1.3}, is bounded on
$L^p, 1<p<\infty.$

%%%%%%%%%%%%%%%%%%%%%%%%%%%%%%%%%%%%%%%%%%%%%%%%%%%%%%%%%%%%%%%%%%%%%%%%%%%%%%%%%%%%%%%%%
%%%%%%%%%%%%%%%%%%%%%%%%%%%%%%%%%%%%%%%%%%%%%%%%%%%%%%%%%%%%%%%%%%%%%%%%%%%%%%%%%%%%%%%%%
\section{$L^2$ boundedness }

The main task of this section is to provide proofs of Theorem \ref{Thm 1.1} and Corollary \ref{Cor 1.2}.
Before proving Theorem \ref{Thm 1.1}, we first show the following simple result which will be used
frequently below.
\begin{lem} \label{Add lem1}
For any $f(x)\in L_{\text{loc}}^1(\Bbb R)$ and $N>8$, we have
\begin{eqnarray*}
&&\bigg|\int_{8 \leq |x| \leq N} f(x)e^{-ix}dx\bigg| \leq\12
\int_{E_N} |f(x)|dx +\12 \int_{8 \leq |x| \leq N} |f(x)-f(x+\pi)|dx,
\end{eqnarray*}
 where $E_N= \{x:4 \leq |x| \leq 12 \} \cup \{x:N-\pi \leq
|x| \leq N+\pi \}$.
\end{lem}

\begin{proof}
We write
\begin{eqnarray*}
\int_{8 \leq |x| \leq N} f(x)e^{-ix}dx=\int_{8 \leq |x+\pi| \leq N}
f(x+\pi)e^{-i(x+\pi)}dx=-\int_{8 \leq |x+\pi| \leq N}
f(x+\pi)e^{-ix}dx.
\end{eqnarray*}
Therefore,
\begin{eqnarray*}
\bigg|\int_{8 \leq |x| \leq N} f(x)e^{-ix}dx\bigg|
&=&\12 \bigg|\int_{8 \leq |x| \leq N} f(x)e^{-ix}dx-\int_{8 \leq |x+\pi| \leq N}
f(x+\pi)e^{-ix}dx \bigg|\\
&\leq& \12 \bigg|\int_{8 \leq |x| \leq N} (f(x)-f(x+\pi))e^{-ix}dx\bigg|\\
&&+\12 \bigg|\int_{\{x:8 \leq |x| \leq N\} \setminus \{x:8 \leq |x+\pi| \leq N\}}
f(x+\pi)e^{-ix}dx \bigg|\\
&&+ \12 \bigg|\int_{\{x:8 \leq |x+\pi| \leq N\} \setminus \{x:8 \leq |x| \leq N\}}
f(x+\pi)e^{-ix}dx \bigg|\\
&=& \12 \bigg|\int_{8 \leq |x| \leq N} (f(x)-f(x+\pi))e^{-ix}dx\bigg|\\
&&+\12 \bigg|\int_{\{x:8 \leq |x-\pi| \leq N\} \setminus \{x:8 \leq |x| \leq
N\}} f(x)e^{-ix}dx \bigg|\\
&&+ \12 \bigg|\int_{\{x:8 \leq |x| \leq N\} \setminus \{x:8 \leq |x-\pi| \leq N\}}
f(x)e^{-ix}dx \bigg|\\
&\leq&\12 \int_{8 \leq |x| \leq N} |f(x)-f(x+\pi)|dx+\12 \int_{E_N} |f(x)|dx
\end{eqnarray*}
and Lemma \ref{Add lem1} follows.
\end{proof}
We now prove Theorem \ref{Thm 1.1}.
\vskip 0.5cm

\noindent{\it Proof of Theorem \ref{Thm 1.1}.} By the Plancherel
theorem, the $L^2$ boundedness of $\mathcal K_\epsilon^N\ast f$ is
equivalent to $|\widehat{\mathcal K_\epsilon^N}(\chi,\eta,\xi)|$
$\leq A$, where $\widehat{\mathcal K_\epsilon^N}$ is the Fourier
transform of $\mathcal K_\epsilon^N$, $A$ is the constant depending
only on the constant $C$ but not on
$\epsilon=(\epsilon_1,\epsilon_2,\epsilon_3)$, and $N=(N_1,N_2,N_3).$
To obtain such an estimate, we may assume that $\chi$ and $\eta$ are
both positive. Note that
\begin{eqnarray*}
\widehat{\mathcal K_\epsilon^N}(\chi,\eta,\xi)
&=& \int_{\epsilon_3\leq |x_3|\leq N_3}\int_{\epsilon_2 \leq |x_2|\leq
N_1}\int_{\epsilon_1 \leq |x_1|\leq N_1}
\mathcal{K}(x_1,x_2,x_3)e^{-ix\chi}e^{-iy\eta}e^{-iz\xi}dx_1dx_2dx_3\\
&=& \int_{\frac{\epsilon_3}{\chi\eta} \leq |x_3|\leq
\frac{N_3}{\chi\eta}} \int_{\frac{\epsilon_2}{\eta} \leq |x_2|\leq
\frac{N_2}{\eta}}\int_{\frac{\epsilon_1}{\chi} \leq |x_1|\leq
\frac{N_1}{\chi}}
\frac{1}{\chi^2\eta^2}\mathcal{K}\Big(\frac{x_1}{\chi},\frac{x_2}{\eta},\frac{x_3}{\chi\eta}\Big)\\
&&\hskip 6cm\cdot e^{-ix_1}e^{-ix_2}e^{-ix_3\xi/(\chi\eta)}dx_1dx_2dx_3.
\end{eqnarray*}
As remarked above, the assumptions on $\mathcal{K}$ are
invariant in the sense that\break $\delta^2_1\delta^2_2
\mathcal{K}(\delta_1x_1,\delta_2 x_2,$ $\delta_1\delta_2 x_3)$ satisfies the
same assumptions as $\mathcal{K}$ with the same constant $C$,
independent of $\delta_1, \delta_2>0.$ Thus
$\frac{1}{\chi^2\eta^2}\mathcal{K}(\frac{x_1}{\chi},\frac{x_2}{\eta},\frac{x_3}{\chi\eta})$
satisfies all conditions (R) and (C1.a) -- (C1.d) with the same
bounds uniformly for $\chi,\eta.$ Therefore, it suffices to show that
$\widehat{\mathcal K_\epsilon^N}(1,1,\xi)$ is a bounded function
uniformly for $0<\epsilon_1,\epsilon_2, \epsilon_3,
N_1,N_2,N_3<\infty.$ To do this, for simplicity, we set
$\epsilon_4=\epsilon_3|\xi|$ and $N_4=N_3|\xi|$. Without loss of
generality, we may assume that $\epsilon_1,\epsilon_2, \epsilon_4
\leq 8 \leq N_1,N_2,N_4$ since all other cases can be written as a
finite linear combination of these cases and can be handled
similarly.

 The bound of
$\widehat{\mathcal K_\epsilon^N}(1,1,\xi)$ follows from the
regularity and cancellation conditions on $\mathcal K.$ More
precisely, we write
\begin{align*}
\widehat{\mathcal K_\epsilon^N}(1,1,\xi)
&= \int_{\epsilon_3 \leq
|x_3| \leq N_3} \yen \xen\K(x_1,x_2,x_3) e^{-ix_1}e^{-ix_2}e^{-ix_3\xi}dx_1dx_2dx_3\\
&=\zen \int_{\yen}
\xen \Kxi e^{-ix_1}e^{-ix_2}e^{-ix_3}dx_1dx_2dx_3\\
&:=I+II,
\end{align*}
where $I$ is the result of integrating over the set $\big\{ 8\leq
|x_1|\leq N_1, \epsilon_2 \leq |x_2|\leq N_1, \epsilon_4 \leq |x_3|\leq
N_4\big\}$ and $II$ over the set $\big\{ \epsilon_1 \leq
|x_1|<8,\epsilon_2 \leq |x_2|\leq N_1, \epsilon_4 \leq |x_3|\leq
N_4\big\}.$

For term $I,$ using  Lemma \ref{Add lem1} with
$f(x_1)=\yen \zen \Kxi e^{-ix_2}\break \cdot e^{-ix_3 }dx_3dx_2$, we obtain
\begin{align*}
|I|
&\lesssim \int_{E_{N_1}}\bigg|\yen \zen
\Kxi e^{-ix_2}e^{-ix_3 }dx_3dx_2\bigg| dx_1\\
&\quad + \x8n \bigg|\yen \zen \Delta_{x_1,\pi}
\Big(\Kxi \Big) e^{-ix_2}e^{-ix_3 }dx_3dx_2\bigg|dx_1\\
&:= I_1+ I_2.
\end{align*}
Then
\begin{align*}
I_1
&\leq \int_{E_{N_1}}\bigg|  \y8n \zen
\Kxi e^{-ix_2}e^{-ix_3 }dx_3dx_2\bigg| dx_1\\
&\quad +\int_{E_{N_1}}\bigg|  \ye8 \zen
\Kxi e^{-ix_2}e^{-ix_3 }dx_3dx_2\bigg| dx_1\\
&:=I_{1,1}+I_{1,2}.
\end{align*}
To estimate term $I_{1,1},$ using Lemma \ref{Add lem1} with
$f(x_2)=\zen \Kxi e^{-ix_3 }dx_3$ we get
\begin{align*}
{|I_{1,1}|}
&\lesssim \int_{E_{N_1}} \int_{E_{N_2}}\bigg|  \zen
\Kxi e^{-ix_3 }dx_3 \bigg| dx_2dx_1\\
&\quad +\int_{E_{N_1}} \y8n \bigg|  \zen \Delta_{x_2,\pi} \Big( \Kxi \Big) e^{-ix_3
}dx_3 \bigg| dx_2dx_1\\
&\lesssim \int_{E_{N_1}} \int_{E_{N_2}}  \zen
\frac{1}{|x| |x_2| |x_3| \big(|\frac{x_1x_2\xi}{x_3}|+|\frac{x_3}{x_1x_2\xi}|
\big)^{\tb}}dx_3  dx_2dx\\
&\quad +\int_{E_{N_1}} \y8n  \zen
\frac{1}{|x_1| |x_2|^{\ta+1} |x_3|
\big(|\frac{x_1x_2\xi}{x_3}|+|\frac{x_3}{x_1x_2\xi}| \big)^{\tb}}dx_3  dx_2dx_1\\
&\lesssim 1,
\end{align*}
where we use the condition (R) above on $\mathcal{K}$ with
$\alpha=\beta=\gamma=0$ and $\alpha=0, \beta=1,\gamma=0,$
respectively.

To handle term $I_{1,2}$, we write
\begin{align*}
I_{1,2}
&\leq \int_{E_{N_1}}\bigg|  \ye8 \z8n
\Kxi e^{-ix_2}e^{-ix_3 }dx_3dx_2\bigg| dx_1\\
&\quad + \int_{E_{N_1}}\bigg|  \ye8 \ze8
\Kxi e^{-ix_2}e^{-ix_3 }dx_3dx_2\bigg| dx_1\\
&:=I_{1,2,1}+I_{1,2,2}.
\end{align*}
By Lemma \ref{Add lem1} with $f(x_3)=\ye8  \Kxi e^{-ix_2} dx_2$, we get
\begin{align*}
|I_{1,2,1}|
&\lesssim \int_{E_{N_1}}   \int_{E_{N_2}}\bigg| \ye8
\Kxi e^{-ix_2}dx_2\bigg| dx_3dx_1\\
&\quad +  \int_{E_{N_1}}   \z8n \bigg| \ye8 \Delta_{x_3,\pi} \Big(\Kxi \Big)
e^{-ix_2}dx_2\bigg| dx_3dx_1\\
&\lesssim \int_{E_{N_1}}
\int_{E_{N_2}} \ye8 \frac{1}{|x_1| |x_2| |x_3|
\big(|\frac{x_1x_2\xi}{x_3}|+|\frac{x_3}{x_1x_2\xi}|
\big)^{\tb}} dx_2 dx_3dx_1\\
&\quad + \int_{E_{N_1}} \z8n \ye8 \frac{1}{|x_1| |x_2| |x_3|^{1+\ta}
\big(|\frac{x_1x_2\xi}{x_3}|+|\frac{x_3}{x_1x_2\xi}| \big)^{\tb}} dx_2 dx_3 dx_1\\
&\lesssim 1,
\end{align*}
where we use the regularity condition (R) above with
$\alpha=\beta=\gamma=0$ and $\alpha=\beta=0, \gamma=1,$
respectively.

To estimate $I_{1,2,2},$ we note that
\begin{align*}
I_{1,2,2}
&\leq \int_{E_{N_1}}\bigg|  \ye8 \ze8
\Kxi (e^{-ix_2}e^{-ix_3 }-1)dx_3dx_2\bigg| dx_1\\
&\quad +\int_{E_{N_1}}\bigg|  \ye8 \ze8 \Kxi dx_3dx_2\bigg| dx_1 \\
&\lesssim \int_{E_{N_1}} \ye8 \ze8 \frac{|x_2|+|x_3|}{|x_1| |x_2| |x_3|
\big(|\frac{x_1 x_2\xi}{x_3}|+|\frac{x_3}{x_1 x_2\xi}| \big)^{\tb}}dx_3dx_2dx_1  +
\int_{E_{N_1}} \frac{1}{|x_1|} dx_1 \\
&\lesssim 1,
\end{align*}
 where we use the condition (R) with $\alpha=\beta=\gamma=0,$ and  the cancellation condition (C1.c) with
$\alpha=0.$

Next we consider $I_2$. Set $I_{2,1}$ and $I_{2,2}$ to be
$$
I_{2,1}= \x8n \bigg|\y8n \zen \Delta_{x_1,\pi} \Big(\Kxi \Big) e^{-ix_2}e^{-ix_3
}dx_3dx_2\bigg|dx_1
$$
and
$$
I_{2,2}= \x8n \bigg|\ye8 \zen \Delta_{x_1,\pi} \Big(\Kxi \Big) e^{-ix_2}e^{-ix_3
}dx_3dx_2\bigg|dx_1.
$$
Then $I_2\leq |I_{2,1}|+|I_{2,2}|.$ Similarly, applying Lemma
\ref{Add lem1} with $$f(x_2)=  \zen \Delta_{x_1,\pi} \Big(\Kxi \Big) e^{-ix_3 }dx_3,$$
we obtain
\begin{align*}
|I_{2,1}|
&\lesssim \x8n \int_{E_{N_2}}\bigg| \zen \Delta_{x_1,\pi} \Big(\Kxi
\Big) e^{-ix_3 }dx_3\bigg|dx_2dx_1\\
&\quad +\x8n \y8n \bigg| \zen \Delta_{x_2,\pi} \Delta_{x_1,\pi}
\Big(\Kxi \Big) e^{-ix_3 }dx_3\bigg|dx_2dx_1\\
&\lesssim \x8n \int_{E_{N_2}}\zen \frac{1}{|x_1|^{1+\ta} |x_2| |x_3|
\big(|\frac{x_1x_2\xi}{x_3}|+|\frac{x_3}{x_1x_2\xi}| \big)^{\tb}}dx_3dx_2dx_1\\
&\quad + \x8n \y8n \zen \frac{1}{|x_1|^{1+\ta} |x_2|^{1+\ta} |x_3|
\big(|\frac{x_1x_2\xi}{x_3}|+|\frac{x_3}{x_1x_2\xi}| \big)^{\tb}}dx_3dx_2dx_1\\
&\lesssim 1,
\end{align*}
where we use the condition (R) above with $\alpha=1,\beta=\gamma=0$
and $\alpha=\beta=1, \gamma=0,$ respectively.

For term $I_{2,2},$  note that
\begin{align*}
I_{2,2}
&\leq \x8n \bigg|\ye8 \z8n \Delta_{x_1,\pi} \Big(\Kxi \Big)
e^{-ix_2}e^{-ix_3 }dx_3dx_2\bigg|dx_1 \\
&\quad+  \x8n \bigg|\ye8 \ze8 \Delta_{x_1,\pi} \Big(\Kxi \Big) e^{-ix_2}e^{-ix_3
}dx_3dx_2\bigg|dx_1\\
&:= I_{2,2,1}+ I_{2,2,2}.
\end{align*}
By Lemma \ref{Add lem1} with $f(x_3)=\ye8  \Delta_{x_1,\pi} \Big(\Kxi \Big)
e^{-ix_2}dx_2$, we have
\begin{align*}
|I_{2,2,1}|
&\lesssim \x8n \int_{E_{N_4}} \bigg|\ye8  \Delta_{x_1,\pi}
\Big(\Kxi \Big) e^{-ix_2}dx_2\bigg|dx_3dx_1\\
&\quad+\x8n \z8n \bigg|\ye8  \Delta_{x_3,\pi} \Delta_{x_1,\pi}
\Big(\Kxi \Big) e^{-ix_2}dx_2\bigg|dx_3dx_1\\
&\lesssim \x8n \int_{E_{N_4}} \ye8 \frac{1}{|x_1|^{1+\ta} |x_2| |x_3|
\big(|\frac{x_1x_2\xi}{x_3}|+|\frac{x_3}{x_1x_2\xi}| \big)^{\tb}}dx_2dx_3dx_1\\
&\quad +\x8n \z8n \ye8 \frac{1}{|x_1|^{1+\ta} |x_2| |x_3|^{1+\ta}
\big(|\frac{x_1x_2\xi}{x_3}|+|\frac{x_3}{x_1x_2\xi}| \big)^{\tb}}dx_2dx_3dx_1\\
&\lesssim
1,
\end{align*}
where we use the conditions (R) above with $\alpha=1,
\beta=\gamma=0$ and $\alpha=\gamma=1, \beta=0,$ respectively.

The estimate for term $I_{2,2,2}$ follows from a similar way as term
$I_{1,2,2}.$ Indeed,
\begin{align*}
I_{2,2,2}
&= \x8n \bigg|\ye8 \ze8 \Delta_{x_1,\pi} \Big(\Kxi \Big)
(e^{-ix_2}e^{-ix_3 }-1)dx_3dx_2\bigg|dx_1\\
&\quad+  \x8n \bigg|\ye8 \ze8 \Delta_{x_1,\pi} \Big(\Kxi \Big) dx_3dx_2\bigg|dx_1\\
&\lesssim \x8n \ye8 \ze8 \frac{|x_2|+|x_3|}{|x_1|^{1+\ta} |x_2| |x_3|
\big(|\frac{x_1x_2\xi}{x_3}|+|\frac{x_3}{x_1x_2\xi}| \big)^{\tb}}dx_2dx_3dx_1\\
&\quad+\x8n \frac{1}{|x_1|^{1+\ta}}dx_1\\
&\lesssim 1,
\end{align*}
where we use the  condition (R) with $\alpha=1, \beta=\gamma=0$ and
the condition (C1.c), respectively.

Now we turn to the estimate for term $II.$ We first write
\begin{align*}
II&= \zen \yen \xe8 \Kxi
(e^{-ix_1}-1)e^{-ix_2}e^{-ix_3 }dx_1dx_2dx_3 \\
&\quad +\zen \yen \xe8 \Kxi e^{-ix_2}e^{-ix_3
}dx_1dx_2dx_3\\
&:=II_1 +II_2.
\end{align*}
We further write
\begin{align*}
|II_1|
&=\zen \y8n \xe8 \Kxi
(e^{-ix_1}-1)e^{-ix_2}e^{-ix_3 }dx_1dx_2dx_3\\
&\quad +\zen \ye8 \xe8 \Kxi
(e^{-ix_1}-1)e^{-ix_2}e^{-ix_3 }dx_1dx_2dx_3\\
&:=II_{1,1}+ II_{1,2}.
\end{align*}
For term $II_{1,1},$ using Lemma \ref{Add lem1} with $f(x_2)=\zen \xe8
\Kxi (e^{-ix_1}-1)e^{-ix_3 }dx_1dx_3$, we obtain
\begin{align*}
|II_{1,1}|
&\lesssim \int_{E_{N_2}}\bigg|\zen \xe8 \Kxi
(e^{-ix_1}-1)e^{-ix_3 }dx_1dx_3\Bigg|dx_2\\
&\quad+\hskip -0.1cm\y8n \bigg| \hskip -0.1cm\zen \hskip -0.1cm\xe8 \hskip -0.1cm\Delta_{x_2,\pi} \Big(\Kxi \Big)
(e^{-ix_1}-1)e^{-ix_3 }dx_1dx_3\Bigg|dx_2\\
&\lesssim  \int_{E_{N_2}}\zen \xe8  \frac{1}{ |x_2| |x_3|
\big(|\frac{x_1x_2\xi}{x_3}|+|\frac{x_3}{x_1x_2\xi}| \big)^{\tb}}dx_1dx_3dx_2\\
&\quad +\y8n \zen \xe8  \frac{1}{ |x_2|^{1+\ta} |x_3|
\big(|\frac{x_1x_2\xi}{x_3}|+|\frac{x_3}{x_1x_2\xi}| \big)^{\tb}}dx_1dx_3dx_2,
\end{align*}
where we use the condition (R) above for $\alpha=\beta=\gamma=0$ and
$\beta=1, \alpha=\gamma=0,$ respectively.

Similarly,
\begin{align*}
II_{1,2}
&=\z8n \ye8 \xe8 \Kxi (e^{-ix_1}-1)e^{-ix_2}e^{-ix_3 }dx_1dx_2dx_3\\
&\quad+ \ze8 \ye8 \xe8 \Kxi (e^{-ix_1}-1)e^{-ix_2}e^{-ix_3 }dx_1dx_2dx_3 \\
&:=II_{1,2,1}+ II_{1,2,2}.
\end{align*}
The bounds of $II_{1,2,1}$ and $II_{1,2,2}$, we follow from the
similar estimates as terms $I_{2,2,1}$ and $I_{2,2,2},$
respectively.

Finally, we estimate term $II_2.$ Denote $II_2=II_{2,1}+II_{2,2},$ where
$$II_{2,1}=\zen \y8n \xe8 \Kxi e^{-ix_2}e^{-ix_3 }dx_1dx_2dx_3$$
and
$$II_{2,2}=\zen \ye8 \xe8 \break \Kxi e^{-ix_2}e^{-ix_3 }dx_1dx_2dx_3.$$
Note that
\begin{align*}
|II_{2,1}|
&=\z8n \y8n \xe8 \Kxi e^{-ix_2}e^{-ix_3 }dx_1dx_2dx_3 \\
&\quad +\ze8 \y8n \xe8 \Kxi e^{-ix_2}e^{-ix_3 }dx_1dx_2dx_3. \\
&:=II_{2,1,1}+II_{2,1,2}.
\end{align*}
Applying Lemma \ref{Add lem1} with $f(x_2)=\z8n  \xe8 \Kxi e^{-ix_3
}dx_1dx_3$ first, then $f(x_3)=  \xe8 \Kxi dx_1$, and combining with the
condition (R), we obtain
\begin{align*}
|II_{2,1,1}|
&\lesssim \int_{E_{N_2}}\bigg| \z8n  \xe8 \Kxi e^{-ix_3}dx_1dx_3 \bigg|dx_2 \\
&\quad +\y8n \bigg| \z8n  \xe8 \Delta_{x_2,\pi} \Big( \Kxi \Big) e^{-ix_3}dx_1dx_3 \bigg|dx_2  \\
&\lesssim \int_{E_{N_2}} \int_{E_{N_4}} \bigg| \xe8 \Kxi e^{-ix_3}dx_1 \bigg| dx_3dx_2\\
&\quad + \int_{E_{N_2}} \z8n \bigg| \xe8 \Delta_{x_3,\pi} \Big(\Kxi \Big) e^{-ix_3}dx_1 \bigg| dx_3dx_2\\
&\quad +\y8n  \int_{E_{N_4}} \bigg| \xe8 \Delta_{x_2,\pi} \Big( \Kxi \Big) dx_1\bigg| dx_3dx_2 \\
&\quad +\y8n  \z8n \bigg| \xe8 \Delta_{x_3,\pi} \Delta_{x_2,\pi} \Big( \Kxi \Big) dx_1\bigg| dx_3dx_2 \\
&\lesssim 1.
\end{align*}

To estimate $II_{2,1,2},$ inserting $e^{-ix_3 }=[e^{-ix_3 }-1]+1$ and
then applying Lemma \ref{Add lem1}, we get
\begin{align*}
II_{2,1,2}
&\lesssim \bigg|\y8n  \xe8 \ze8   \Kxi e^{-ix_2}(e^{-ix_3}-1) dx_3dx_1 dx_2\bigg|\\
&\quad +\bigg| \y8n   \xe8 \ze8  \Kxi e^{-ix_2} dx_3dx_1dx_2 \bigg|\\
&\lesssim \int_{E_{N_2}} \xe8 \ze8    \Big| \Kxi \Big| |x_3|  dx_3dx_1dx_2 \\
&\quad + \y8n \xe8 \ze8     \Big| \Delta_{x_2,\pi} \Big(\Kxi\Big) \Big| |x_3| dx_3dx_1dx_2  \\
&\quad + \int_{E_{N_2}}\bigg| \xe8 \ze8 \Kxi dx_3dx_1\bigg| dx_2 \\
&\quad +\y8n \bigg| \xe8 \ze8 \Delta_{x_2,\pi} \Big( \Kxi \Big) dx_3dx_1\bigg| dx_2.
\end{align*}
The required bound then follows from the conditions (R) for the
first two integrals while the condition (C1.d)  for the last two
integrals.

For $II_{2,2}$, splitting the set $\{\epsilon_4 \leq |x_3| \leq N_4\}$
into two parts $\{\epsilon_4 \leq |x_3| \leq 8\}$ and $\{8 \leq |x_3|
\leq N_4\}$, and inserting $e^{-ix_2}e^{-ix_3 }= (e^{-ix_2}-1)(e^{-ix_3}-1)
+(e^{-ix_2} -1) + (e^{-ix_3}-1)+ 1$ for the integral over the first set
and $e^{-ix_2}e^{-ix_3 }= (e^{-ix_2}-1)e^{-ix_3}+e^{-ix_3}$ for the integral
over the second set, we obtain
\begin{align*}
II_{2,2}
&\lesssim \bigg|\ze8 \ye8 \xe8 \Kxi (e^{-ix_2}-1)(e^{-ix_3
}-1)dx_1dx_2dx_3\bigg|\\
 %%%%%%%%term1
&\quad+ \bigg|\ze8 \ye8 \xe8 \Kxi (e^{-ix_2}-1)dx_1dx_2dx_3\bigg|
 %%%%%%%%term2
\\&\quad+\bigg|\ze8 \ye8 \xe8 \Kxi (e^{-ix_3
}-1)dx_1dx_2dx_3\bigg|\\
 %%%%%%%%% term3
&\quad+\bigg|\ze8 \ye8 \xe8 \Kxi dx_1dx_2dx_3\bigg|
 %%%%%%%%% term4
\\&\quad+ \bigg|\z8n \ye8 \xe8 \Kxi (e^{-ix_2}-1)e^{-ix_3 }dx_1dx_2dx_3\bigg|
%%%%%%%%term5
\\&\quad+\bigg|\z8n \ye8 \xe8 \Kxi e^{-ix_3
}dx_1dx_2dx_3\bigg|.
%%%%%%%%term6
\end{align*}
The first four items follow from the conditions (R),(C1.d),(C1.b),
and (C1.a), respectively. To estimate the fifth and sixth terms,
we apply Lemma \ref{Add lem1} to get
\begin{align*}
&\bigg|\z8n \ye8 \xe8 \Kxi (e^{-ix_2}-1)e^{-ix_3 }dx_1dx_2dx_3\bigg|\\
%%%%%%%%term5
&\quad +\bigg|\z8n \ye8 \xe8 \Kxi e^{-ix_3
}dx_1dx_2dx_3\bigg| %%%%%%%%term6
\\&\lesssim
\int_{E_{N_4}} \bigg| \ye8 \xe8 \Kxi (e^{-ix_2}-1)dx_1dx_2\bigg|dx_3
%%%%%%%%
\\&\quad + \z8n \bigg| \ye8 \xe8 \Delta_{x_3,\pi} \Big( \Kxi \Big) (e^{-ix_2}-1)dx_1dx_2\bigg|dx_3
%%%%%%%%
\\&\quad +\int_{E_{N_4}} \bigg|\ye8 \xe8 \Kxi dx_1dx_2\bigg| dx_3
%%%%%%%%
\\&\quad +\z8n \bigg|\ye8 \xe8 \Delta_{x_3,\pi} \Big(\Kxi \Big) dx_1dx_2\bigg| dx_3\\
&\lesssim 1,
%%%%%%%%
\end{align*}
where we use the condition (R) for the first two terms and (C1.b)
for the last two terms above. Thus these estimates yield the bound
of $II_{2,2}$ and hence the required bound for term $II.$ The $L^2$
boundedness of $\mathcal K_\epsilon^N\ast f$ follows. \qed
\vskip 0.5cm

\noindent{\it Proof of Corollary \ref{Cor 1.2}.} It
suffices to show that $\mathcal K_\epsilon^N\ast f$ converges in
$L^2(\Bbb R^3)$, as
$\epsilon_1,\epsilon_2,\epsilon_3\to 0$ and $N_1,N_2,N_3\to \infty,$ for a dense subset of $L^2(\Bbb R^3)$.
For this purpose, we consider smooth functions $f$ having compact support. We
may assume that $\epsilon_1,\epsilon_2,\epsilon_3<1$ and
$N_1,N_2,N_3>1.$

We write $\iiint_{\Bbb R^3}\mathcal
K_\epsilon^N(u)f(x-u))du$ as a sum of eight terms;
that is, the integrals over the sets
(i) $|u_1|\leq 1, |u_2|\leq 1, |u_3|\leq 1$;
(ii) $|u_1|\leq 1, |u_2|\leq 1, |u_3|> 1$;
(iii) $|u_1|\leq 1, |u_2|>1, |u_3|\leq 1$;
(iv) $|u_1|\leq 1, |u_2|> 1, |u_3|> 1$;
(v) $|u_1|> 1,|u_2|\leq 1, |u_3|\leq 1$;
(vi) $|u_1|> 1, |u_2|\leq 1, |u_3|> 1$;
(vii) $|u_1|> 1, |u_2|> 1, |u_3|\leq 1$;
(viii) $|u_1|> 1, |u_2|> 1, |u_3|> 1$.
Inserting
\begin{align*}
f(x-u)
&=[f(x_1-u_1,x_2-u_2,x_3-u_3)-f(x_1,x_2-u_2,x_3-u_3)-f(x_1-u_1,x_2,x_3)\\
&\quad +f(x_1,x_2,x_3)]+[f(x_1-u_1,x_2,x_3)-f(x_1,x_2,x_3)]\\
&\quad+[f(x_1,x_2-u_2,x_3-u_3)-f(x_1,x_2-u_2,x_3)]\\
&\quad +[f(x_1,x_2-u_2,x_3)-f(x_1,x_2,x_3)]+f(x_1,x_2,x_3)
\end{align*}
into the first term
\begin{eqnarray*}
\int_{\epsilon_1 \leq |u_1|\leq 1}\int_{\epsilon_2 \leq |u_2|\leq
1}\int_{\epsilon_3 \leq |u_3|\leq 1}\mathcal K(u_1,u_2,u_3)f(x_1-u_1,x_2-u_2,x_3-u_3)
du_1du_2du_3
\end{eqnarray*}
yields five integrals. In view of the conditions of $f$ and the
condition (R) on $\mathcal K,$ the first integral is dominated by
\begin{align*}
&F_1(x_1)F_2(x_2)F_3(x_3)\int_{ |u_1|\leq 1}\int_{|u_2|\leq 1}\int_{|u_3|\leq
1} \frac{1}{|u_1| |u_2|
|u_3|}\Big(\Big|\frac{u_1u_2}{u_3}\Big|+\Big|\frac{u_3}{u_1u_2}\Big|
\Big)^{-\tb}\\
&\hskip 7cm \times |u_1|(|u_2|+|u_3|)du_1du_2du_3,
\end{align*}
where $F_1(x_1)$, $F_2(x_2)$ and $F_3(x_3)$ are bounded functions with
bounded supports. Thus, as $\epsilon_1, \epsilon_2,\epsilon_3\to 0,$
the limit of the first integral exists for each $x_1, x_2,$ and $x_3$ and,
moreover, is dominated by a fixed bounded function with compact
support. Therefore, the first integral converges in $L^2$ as
$\epsilon_1, \epsilon_2,\epsilon_3\to 0.$ The third integral can be
handled by the same way. To see the second integral, by the
condition (C1.c) and the assumption on $\mathcal K$ we observe that
the limit $\int_{\epsilon_2\leq |u_2|\leq 1}\int_{\epsilon_3\leq
|u_3|\leq 1}\mathcal K(u_1,u_2,u_3)du_2du_3$ exists as $\epsilon_2,
\epsilon_3\to 0,$ and is dominated by $C|u_1|^{-1}.$ This fact
together with the smoothness condition on $f$ implies the second
integral converges in $L^2$ as $\epsilon_1, \epsilon_2,\epsilon_3\to
0$ and the limit is dominated by a bounded function with compact
support. Similarly, the required results for the fourth and the last
integrals follow from the conditions (C1.d) and (C1.a),
respectively, together with the assumptions on $\mathcal K.$

Note that in fact $\mathcal K(u)$ is integrable over the sets
(ii) $|u_1|\leq 1, |u_2|\leq 1, |u_3|\geq 1$ and
(vii) $|u_1|\geq 1, |u_2|\geq 1, |u_3|\leq 1$.
Thus we have all the required results over these two sets.

Observe that
\begin{align*}
\int_{\Bbb R}\int_{\Bbb R}\int_{|u_3|\geq 1}|\mathcal K(u)f(x-u)|du
&\lesssim \int_{\Bbb R}\int_{\Bbb R}\int_{|u_3|\geq 1}
\frac{1}{|u_1| |u_2| |u_3|} \Big(\Big|\frac{u_1u_2}{u_3}\Big|+\Big|\frac{u_3}{u_1u_2}\Big|\Big)^{-\tb} \\
&\quad\times \frac{1}{(1+|x_1-u_1|)^2 (1+|x_2-u_2|)^2 (1+|x_3-u_3|)^2}du,
\end{align*}
which belongs to $L^2(\Bbb R^3).$ This implies the required results
over the corresponding sets (iv), (vi) and (viii).

To handle the integral over the set (iii) $|u_1|\leq 1, |u_2|\geq 1, |u_3|\leq 1$, inserting
\begin{align*}
f(x_1-u_1,x_2-u_2,x_3-u_3)
&=[f(x_1-u_1,x_2-u_2,x_3-u_3)-f(x_1,x_2-u_2,x_3-u_3)]\\
&\quad +[f(x_1,x_2-u_2,x_3-u_3)-f(x_1,x_2-u_2,x_3)]\\
&\quad +f(x_1,x_2-u_2,x_3)
\end{align*}
yields three integrals over the set (iii). The first two
integrals, by the condition (R) and the smoothness of $f,$ are
dominated by
\begin{eqnarray*}
F_1(x_1)F_3(x_3)\int_{ |u_1|\leq 1}\int_{|u_2|\geq 1}\int_{|u_3|\leq 1}
\frac{1}{|u_2| |u_3|}
\Big(\Big|\frac{u_1u_2}{u_3}\Big|+\Big|\frac{u_3}{u_1u_2}\Big|\Big)^{-\tb}
\frac{1}{(1+|x_2-u_2|)^2}du_1du_2du_3
\end{eqnarray*}
and
\begin{eqnarray*}
F_1(x_1)F_3(x_3)\int_{ |u_1|\leq 1}\int_{|u_2|\geq 1}\int_{|u_3|\leq 1}
\frac{1}{|u_1| |u_2|}
\Big(\Big|\frac{u_1u_2}{u_3}\Big|+\Big|\frac{u_3}{u_1u_2}\Big|\Big)^{-\tb}\frac{1}{(1+|x_2-u_2|)^2}du_1du_2du_3,
\end{eqnarray*}
where $F_1(x_1)$ and $F_3(x_3)$ are bounded functions with bounded
supports. Thus, we obtain a domination, independent of
$\epsilon_1,\epsilon_3$ and $N_2,$ by a function which belongs to
$L^2(\Bbb R^3),$ so the limits as $\epsilon_1,\epsilon_3 \to 0$ and
$N_2\to\infty$ exist. Condition (C1.d) with $\beta=0$
yields that the last integral is bounded by
\begin{eqnarray*}
F_1(x_1)F_3(x_3)\int_{|u_2|\geq 1}\frac{1}{|u_2| (1+|x_2-u_2|)^2}du_2,
\end{eqnarray*}
which belongs to $L^2(\Bbb R^3)$ and the limit as
$\epsilon_1,\epsilon_3 \to 0$ and $N_2\to\infty$ exists.

Finally, for the integral over the set (v) $|u_1|\geq 1, |u_2|\leq 1,|u_3|\leq 1$, by inserting
\begin{align*}
f(x_1-u_1,x_2-u_2,x_3-u_3)
&= [f(x_1-u_1,x_2-u_2,x_3-u_3)-f(x_1-u_1,x_2-u_2,x_3)]\\
&\quad +[f(x_1-u_1,x_2-u_2,x_3)-f(x_1-u_1,x_2,x_3)]\\
&\quad +f(x_1-u_1,x_2,x_3)
\end{align*}
and then applying the condition (R) for the first two integrals and
(C1.c) with $\alpha=0$ on the last integral, this integral is
dominated by
\begin{eqnarray*}
F_2(x_2)F_3(x_3)\int_{|u_1|\geq 1}\frac{1}{|u_1| (1+|x_1-u_1|)^2}du_1,
\end{eqnarray*}
where $F_2(x_2)$ and $F_3(x_3)$ are bounded functions with bounded
supports. The existence of the limit is concluded. The $L^2$
boundedness of $\mathcal K\ast f$ then follows from Theorem \ref{Thm
1.1}. \qed

\begin{rem}
As mentioned early in section 1, we have incidentally shown that $\mathcal K_\epsilon^N\ast f$
converges in $L^p$ norm and almost everywhere as
$\epsilon_1,\epsilon_2,\epsilon_3\to 0$ and $N_1,N_2,N_3\to\infty$
whenever $f$ is a smooth function with compact support. We also
point out that the condition (C1.b) is not used in the proof of
Corollary \ref{Cor 1.2}.
\end{rem}

%%%%%%%%%%%%%%%%%%%%%%%%%%%%%%%%%%%%%%%%%%%%%%%%%%%%%%%%%%%%%%%%%%%%
%%%%%%%%%%%%%%%%%%%%%%%%%%%%%%%%%%%%%%%%%%%%%%%%%%%%%%%%%%%%%%%%%%%%
\section{$L^p$ estimates for $1<p<\infty$}

In this section, we will prove Theorem \ref{Thm 1.3} and Theorem \ref{Thm 1.4}. The
main tools to show the $L^p$, $1<p<\infty,$ estimates are
\begin{itemize}
\item the $L^2$ boundedness of $\mathcal K\ast f;$
\item the Littlewood--Paley theory associated with Zygmund dilation;
\item the almost orthogonality argument.
\end{itemize}
We first recall the Littlewood--Paley theory. As mentioned in
section 1, to handle the $L^p$, $1<p<\infty,$ boundedenss of
operators, one only needs the continuous Littlewood-Paley square
function. To do this, let ${\mathcal S} (\Bbb R^i)$ denote the
Schwartz class in $\Bbb R^i$, $i=1, 2, 3$. We construct a function
defined on $\Bbb R^3$ by
\begin{equation}\label{eq 3.1}
\phi (x_1,x_2,x_3)=\phi^{(1)}(x)\phi^{(2)}(x_2,x_3),
\end{equation}
where $\phi^{(1)}\in {\mathcal S}(\Bbb R)$, $\phi^{(2)}\in {\mathcal
S}(\Bbb R^2)$ with the supports contained in the unit ball centered
at the origin in $\Bbb R^3,$ and satisfy
\begin{equation}\label{eq 3.2}
\sum _{j \in \Bbb Z}\vert \widehat{\phi^{(1)}}(2^{j}\xi _1
)\vert^2=1 \qquad\text{for all}\ \ \xi _1\in {\Bbb R}\backslash
\lbrace 0\rbrace,
\end{equation}
\begin{equation}\label{eq 3.3}
\sum _{k\in \Bbb Z}\vert \widehat{ \phi ^{(2)}}(2^{k} \xi_2,
2^{k}\xi_3 )\vert ^2= 1\qquad\text{for all}\ \ (\xi_2, \xi_3) \in
{\Bbb R^2}\backslash \lbrace (0, 0)\rbrace,
\end{equation} and the moment conditions
\begin{equation}\label{eq 3.4}
\int_{\Bbb R}x_1^{\alpha}\phi^{(1)}(x_1)dx_1=\int_{\Bbb R^2}
x_2^{\beta}x_3^{\gamma}\phi^{(2)}(x_2,x_3)dx_2dx_3=0 \qquad\text{for}\ \ 0\leq
\alpha, \beta,\gamma\leq 10.
\end{equation}

For $f\in L^p$, $1<p<\infty$,
{\it the continuous Littlewood--Paley square function $g_{\frak z}^c(f)$ of $f$ associated with the Zygmund dilation}
is defined by
\begin{equation*}
g_{\frak z}^c(f)(x) =\bigg\lbrace \sum _{j,k\in \Bbb Z}\vert (\phi
_{j,k}*f)(x)\vert ^2\bigg\rbrace^{{{1}\over{2}}},
\end{equation*}
where
\begin{equation}\label{eq 3.6}
\phi_{j,k}(x_1,x_2,x_3):=2^{-2(j+k)}\phi^{(1)}(2^{-j}x_1 )
\phi^{(2)}(2^{-k}x_2, 2^{-(j+k)}x_3).
\end{equation}
By taking the Fourier transform, it is easy to see that the
following Calder\'on's identity
\begin{equation}\label{eq 3.7}
f(x) =\sum _{j,k\in \Bbb Z} (\phi_{j,k}*\phi _{j,k}*f)(x)
\end{equation}
holds on $L^2(\Bbb R^3)$.
Using the $L^p$ boundedness of operators for $1<p<\infty$ in
\cite{RS}, as mentioned in section 1, we have
$$\bigg\|\sum _{(j,k)\in F} \epsilon(j,k)(\phi_{j,k}*f)\bigg\|_p\leq C\|f\|_p$$
for every sequence $\epsilon(j,k),$ taking the values 1 and $-1$,
where $F$ is any finite subset of $j, k\in \Bbb Z$.
By Khinchin's well-known inequality,
$$\|g_{\frak z}^c(f) \| _p\leq C_2\|f\|_p\qquad\text{for}\ \ 1<p<\infty.$$
This estimate together with Calder\'on's identity
on $L^2$ allows us to obtain the $L^p$ estimates of $g_{\frak z}^c$  for
$1<p<\infty$. Namely, there exist constants $C_1$ and $C_2$ such
that, for $1<p<\infty,$
\begin{equation}\label{eq 3.8}
C_1\Vert f \Vert _p\leq \Vert g_{\frak z}^c(f) \Vert _p\leq C_2\Vert
f\Vert_p.
\end{equation}

Now we turn to the proof of Theorem \ref{Thm 1.3}. First note that $\mathcal
K$ satisfies the conditions (C1.a) -- (C1.d) since (C2.b) implies
(C1.b) and (C1.d), as mentioned in section 1. Therefore, by Corollary
\ref{Cor 1.2}, the operator $\mathcal K\ast
f=\lim_{\epsilon_1,\epsilon_2,\epsilon_3\to 0 \atop N_1,N_2,N_3\to
\infty} \mathcal K_\epsilon^N\ast f$ is bounded on $L^2(\Bbb R^3).$
To obtain the $L^p$ boundedness of $\mathcal K\ast f,$ it suffices
to show this for all $f\in L^2\cap L^p$ since the subspace $L^2\cap
L^p$ is dense in $L^p$, $1<p<\infty.$ By the $L^p$ estimates of the
Littlewood--Paley square function given in (\ref{eq 3.8}), the $L^p$
boundedness of $\mathcal K\ast f$ will follow from the estimate
\begin{equation}\label{eq 3.9}
\Vert g_{\frak z}^c(\mathcal K\ast f)\Vert_p \lesssim \Vert f\Vert_p.
\end{equation}
To prove (\ref{eq 3.9}) for all $f\in L^2\cap L^p,$ using the fact
that $\mathcal K\ast f$ is bounded on $L^2,$ as mentioned above, and
Calder\'on's identity on $L^2$ given in (\ref{eq 3.7}), we write
$$\phi _{j,k}\ast (\mathcal K\ast f)(x_1,x_2,x_3)=\sum _{j',k'\in \Bbb Z}
(\phi _{j,k}\ast\mathcal K\ast \phi_{j',k'}\ast\phi _{j',k'}\ast
f)(x_1,x_2,x_3).$$ The proof of Theorem \ref{Thm 1.3} now follows from the
following almost orthogonality argument.

\begin{pro}\label{Prop 3.1}
Suppose that $\phi_{j,k}$ is defined as in \eqref{eq 3.6} and $\mathcal K$
is a function on $\Bbb R^3$ satisfying the conditions {\rm(R)} and {\rm(C2.a)}
-- {\rm(C2.c)}. Then, for $\lambda=\12 \min(\ta,\tb),$
\begin{align*}
\big| (\phi_{j,k} \ast \mathcal K \ast \phi_{j',k'})(x_1,x_2,x_3) \big|
&\le C2^{-|j-j'|}2^{-|k-k'|}\frac{2^{-(j\vee j')}}{(1+2^{-(j\vee
j')}|x_1|)^{1+\lambda}}\\
&\quad \times \frac{2^{-(k\vee k')}}{(1+2^{-(k\vee
k')}|x_2|)^{1+\lambda }}\frac{2^{-(j\vee j')-(k\vee
k')}}{(1+2^{-(j\vee j')-(k\vee k')}|x_3|)^{1+\lambda }},
\end{align*}
where the constant C depends only on $\lambda$ and $\mathcal K\ast
f$ is defined for $f\in L^2$ as in\break Corollary \ref{Cor 1.2}, and
$j\vee j'$ means $\max(j, j')$.
\end{pro}

Assuming Proposition \ref{Prop 3.1} for the moment, we then observe
that
$$|[\phi_{j,k}\ast \mathcal K\ast\phi_{j',k'}\ast(\phi_{j',k'}\ast
f)](x_1,x_2,x_3)|\leq C 2^{-|j-j'|}2^{-|k-k'|}\mathcal
M_s(\phi_{j',k'}\ast f)(x_1,x_2,x_3),$$ where $\mathcal M_s$ is the
strong maximal function on $\Bbb R^3.$ H\"older's inequality implies
\begin{eqnarray*}
\|g_{\frak z}^c(\mathcal K\ast f)\|_p &=&\bigg\|\bigg\lbrace
\sum_{j,k}\vert \phi _{j,k}\ast \mathcal K\ast f\vert^2\bigg\rbrace^{{{1}\over{2}}}\bigg\|_p
   \leq C \bigg\|\bigg\{\sum_{j',k'} |\mathcal M_s(\phi_{j',k'}\ast f)|^2\bigg\}^{\frac{1}{2}}\bigg\|_p \\
&\leq& C\bigg\|\bigg\{\sum_{j',k'} |\phi_{j',k'}\ast
f|^2\bigg\}^{\frac{1}{2}
 }\bigg\|_p \leq C \|f\|_p,
\end{eqnarray*}
where we use Fefferman--Stein's vector-valued maximal
inequality and the Littlewood--Paley square function estimate for
$L^p$, $1<p<\infty,$ in the last two inequalities, respectively.

To finish the proof of Theorem \ref{Thm 1.3}, we only need to show Proposition
\ref{Prop 3.1}, whose proof follows from the following lemma.

\begin{lem}\label{Lem 3.2}
Suppose that $\phi^{(1)}$ and $\phi^{(2)}$ satisfy the conditions
\eqref{eq 3.1} -- \eqref{eq 3.4} and $\mathcal K$ is a function on
$\Bbb R^3$ satisfying the conditions {\rm(R)} and {\rm(C2.a)} -- {\rm(C2.c)}. Then,
for $\lambda=\12 \min(\ta,\tb),$
\begin{eqnarray*}
|\mathcal K\ast (\phi^{(1)}\otimes \phi^{(2)})(x_1,x_2,x_3)|\leq
\frac{C_\lambda}{(1+|x_1|)^{1+\lambda} (1+|x_2|)^{1+\lambda}
(1+|x_3|)^{1+\lambda }},
\end{eqnarray*}
where $C_\lambda$ is the constant depending only on $\lambda.$
\end{lem}

\begin{proof} For simplicity, let
$S=\lim_{\epsilon_1,\epsilon_2,\epsilon_3 \to 0 \atop N_1,N_2,N_3
\to \infty}\int_{\epsilon_1 \leq |x-u|\leq N_1}\int_{\epsilon_2 \leq
|x_2-u_2|\leq N_2}\int_{\epsilon_3 \leq |x_3-u_3|\leq N_3} \mathcal
K(x_1-u_1,x_2-u_2,x_3-u_3)\phi^{(1)}(u_1)\phi^{(2)}(u_2,u_3)du_3du_2du_1$. We consider the
following eight cases.

Case 1. $|x_1|\geq 3, |x_2| \geq 3, |x_3| \geq 3$. For this case, we use
the cancellation conditions in (\ref{eq 3.4}) to write
\begin{align*}
S
&=\lim_{\epsilon_1,\epsilon_2,\epsilon_3 \to 0 \atop N_1,N_2,N_3
\to \infty}\int_{\epsilon_1 \leq |x_1-u_1|\leq N_1}\int_{\epsilon_2 \leq
|x_2-u_2|\leq N_2}\int_{\epsilon_3\leq |x_3-u_3|\leq N_3} \big[\mathcal
K(x_1-u_1,x_2-u_2,x_3-u_3)\\
&\hskip 4cm - \mathcal
K(x_1,x_2-u_2,x_3-u_3)-\mathcal K(x_1-u_1,x_2,x_3)+\mathcal K(x_1,x_2,x_3)\big]\\
&\hskip 4cm \times\phi^{(1)}(u_1)\phi^{(2)}(u_2,u_3)du_3du_2du_1.
\end{align*}
Note that $ \mathcal K(x_1-u_1,x_2-u_2,x_3-u_3)-\mathcal K(x_1,x_2-u_2,x_3-u_3)-
\big(\mathcal K(x_1-u_1,x_2,x_3)-\mathcal K(x_1,x_2,x_3)\big) = \Delta_{x_2,-u_2} \Delta_{x_1,-u_1}
\K(x_1,x_2,x_3-u_3)+\Delta_{x_3,-u_3} \Delta_{x_1,-u_1} \K (x_1,x_2,x_3).$ Thus, by the condition (R) with
$\alpha=\beta=1, \gamma=0$ and $\alpha=\gamma=1,\beta=0$, respectively,
\begin{align*}
|S|
&\lesssim \int_{|u_1| \leq 1}\int_{|u_2| \leq 1}\int_{|u_3| \leq 1}
\frac{|u_1|^\ta(|u_2|^\ta+|u_3|^\ta)}{|x_1|^{1+ \ta} |x_2|^{1+\ta}
|x_3|}\Big(\Big|\frac{x_1x_2}{x_3}\Big|+\Big|
\frac{x_3}{x_1x_2}\Big|\Big)^{-\lambda}du_3du_2du_1\\
&\lesssim \frac{1}{(1+|x_1|)^{1+\lambda} (1+|x_2|)^{1+\lambda}
(1+|x_3|)^{1+\lambda }}.
\end{align*}

Case 2. $|x_1|\geq 3, |x_2| \geq 3, |x_3| < 3$. By the cancellation
condition of $\phi^{(1)},$
\begin{align*}
S
&=\lim_{\epsilon_1,\epsilon_2,\epsilon_3 \to 0 \atop N_1,N_2,N_3
\to \infty}\int_{\epsilon_1 \leq |x_1-u_1|\leq N_1}\int_{\epsilon_2 \leq
|x_2-u_2|\leq N_2}\int_{\epsilon_3 \leq
|x_3-u_3|\leq N_3} \Delta_{x_1,-u_1} \mathcal K(x_1,x_2-u_2,x_3-u_3) \\
&\hskip 9.5cm \times\phi^{(1)}(u_1)\phi^{(2)}(u_2,u_3)du_3du_2du_1.
\end{align*}
Therefore, by the condition (R) with $\alpha=1$ and
$\beta=\gamma=0,$ we obtain
\begin{align*}
|S|
&\lesssim \int_{|u_1| \leq 1}\int_{|u_2| \leq 1}\int_{|u_3| \leq 1}
\frac{|u_1|^\ta}{|x_1|^{1+\ta} |x_2-u_2| |x_3-u_3|}\\
&\hskip 4cm\times \Big( \Big|\frac{x_1(x_2-u_2)}{x_3-u_3}\Big| + \Big|\frac{x_3-u_3}{x_1(x_2-u_2)}\Big|\Big)^{-\lambda}du_3du_2du_1\\
&\lesssim \int_{|u_1| \leq 1}\int_{|u_2| \leq 1}\int_{|u_3| \leq 1}
\frac{1}{|x_1|^{1+\ta} |x_2| |x_3-u_3|}\Big|\frac{x_1x_2}{x_3-u_3}\Big|^{-\lambda}du_3du_2du_1\\
&\lesssim \frac{1}{(1+|x_1|)^{1+\lambda} (1+|x_2|)^{1+\lambda}
(1+|x_3|)^{1+\lambda }}.
\end{align*}

Case 3. $|x_1|\geq 3, |x_2| < 3, |x_3|  \geq 3.$ The same expression for
$S$ as in  case 2 yields
\begin{align*}
|S|
&\lesssim \int_{|u_1| \leq 1}\int_{|u_2| \leq 1}\int_{|u_3| \leq 1}
\frac{|u_1|^\ta}{|x_1|^{1+\ta} |x_2-u_2|
|x_3|}\Big|\frac{x_3}{x_1(x_2-u_2)}\Big|^{-\lambda}du_3du_2du_1\\
&\lesssim \frac{1}{(1+|x_1|)^{1+\lambda} (1+|x_2|)^{1+\lambda} (1+|x_3|)^{1+\lambda }}.
\end{align*}

Case 4. $|x_1|\geq 3, |x_2| < 3, |x_3| < 3.$ Using the cancellation
condition of $\phi^{(1)}$, we write
\begin{align*}
S
&= \lim_{\epsilon_1,\epsilon_2,\epsilon_3 \to 0 \atop N_1,N_2,N_3
\to \infty}\int_{\epsilon_1 \leq |x_1-u_1|\leq N_1}\int_{\epsilon_2 \leq
|u_2|\leq 4}\int_{\epsilon_3 \leq |u_3|\leq 4}
\big[\mathcal K(x_1-u_1,u_2,u_3)-\mathcal K(x_1,u_2,u_3)\big]\\
&\hskip 4.5cm\times \phi^{(1)}(u_1)\big(\phi^{(2)}(x_2-u_2,x_3-u_3)-\phi^{(2)}(x_2,x_3)\big)du_3du_2du_1 \\
&\qquad+\lim_{\epsilon_1,\epsilon_2,\epsilon_3 \to 0 \atop N_1,N_2,N_3
\to \infty}\int_{\epsilon_1 \leq |x_1-u_1|\leq N_1}\int_{\epsilon_2 \leq
|u_2|\leq 4}\int_{\epsilon_3 \leq |u_3|\leq
4}\big[\mathcal K(x_1-u_1,u_2,u_3)-\mathcal K(x_1,u_2,u_3)\big]\\
&\hskip 4.5cm\times \phi^{(1)}(u_1)\phi^{(2)}(x_2,x_3)du_3du_2du_1.
\end{align*}
By the condition (R) with $\alpha=1,\beta=\gamma=0$ for the first
integral, and the cancellation condition (C2.c) with $\alpha=1$ for
the second integral,
\begin{align*}
|S|
&\lesssim \int_{|u_1| \leq 1}\int_{|u_2| \leq 4}\int_{|u_3| \leq 4}
\frac{|u_1|^\ta}{|x_1|^{1+\ta} |u_2|
|u_3|}\Big(\Big|\frac{x_1u_2}{u_3}\Big|+\Big|\frac{u_3}{x_1u_2}\Big| \Big)^{-\tb}
(|u_2|+|u_3|)du_3du_2du_1\\
&\quad + \int_{|u_1| \leq 1}
\frac{|u_1|^\ta}{|x_1|^{1+\ta}} du_1\\
&\lesssim \frac{1}{|x_1|^{1+\ta}}\\
&\lesssim
\frac{1}{(1+|x_1|)^{1+\lambda} (1+|x_2|)^{1+\lambda} (1+|x_3|)^{1+\lambda
}}.
\end{align*}

Case 5. $|x_1|< 3, |x_2| \geq 3, |x_3| \geq 3$. Similar to case 4, using
the cancellation condition of $\phi^{(2)},$ we write
\begin{align*}
S
&= \lim_{\epsilon_1,\epsilon_2,\epsilon_3 \to 0 \atop N_1,N_2,N_3
\to \infty}\int_{\epsilon_1 \leq |u_1|\leq 4}\int_{\epsilon_2 \leq
|x_2-u_2|\leq N_2}\\
&\hskip 2.5cm \int_{\epsilon_3 \leq |x_3-u_3|\leq N_3}
\big[\mathcal K(u_1,x_2-u_2,x_3-u_3)-\mathcal K(u_1,x_2,x_3)\big]\\
&\hskip 5cm \times \big(\phi^{(1)}(x_1-u_1)-\phi^{(1)}(x_1)\big)\phi^{(2)}(u_2,u_3)du_3du_2du_1 \\
&\qquad+\lim_{\epsilon_1,\epsilon_2,\epsilon_3 \to 0 \atop N_1,N_2,N_3
\to \infty}\int_{\epsilon_1 \leq |u_1|\leq 4}\int_{\epsilon_2\leq
|x_2-u_2|\leq N_2}\\
&\hskip 3cm \int_{\epsilon_3 \leq |x_3-u_3|\leq
N_3} \big[\mathcal K(u_1,x_2-u_2,x_3-u_3)-\mathcal K(u_1,x_2,x_3)\big]\\
&\hskip 5.5cm\times \phi^{(1)}(x_1)\phi^{(2)}(u_2,u_3)du_3du_2du_1.
\end{align*}
Note that $\mathcal K(u_1,x_2-u_2,x_3-u_3)-\mathcal K(u_1,x_2,x_3)=\Delta_{x_2,-u_2} \K(u_1,x_2,x_3)+
\Delta_{x_3,-u_3} \K(u_1,x_2,x_3)$ and $\mathcal K(u_1,x_2-u_2,x_3-u_3)-\mathcal K(u_1,x_2,x_3)=\Delta_{x_2,-u_2}
\K(u_1,x_2,x_3-u_3) +\Delta_{x_3,-u_3} \K(u_1,x_2,x_3)$. Thus, using the condition (R) on
$\mathcal K$, the smoothness of $\phi^{(1)}$ for the first integral,
and the cancellation conditions (C2.b) with $\beta=1, \gamma=0$ and
$\beta=0, \gamma=1,$ respectively, for the second integral, and
applying the dominated convergence theorem, we obtain
\begin{align*}
|S|
&\lesssim \int_{|u_1|\le 4}\int_{|u_2|\le 1}\int_{|u_3|\le 1}
\bigg(\frac{|u_2|^\ta}{|u_1| |x_2|^{1+\ta} |x_3|}+\frac{|u_3|^\ta}{|u_1| |x_2|
|x_3|^{1+\ta}}\Big)\\
&\hskip 4cm\times \Big(\Big|\frac{u_1x_2}{x_3}\Big|+\Big|\frac{x_3}{u_1x_2}\Big|\Big)^{-\lambda}|u_1|du_3du_2du_1\\
&\quad + \int_{|u_2| \leq 1}\int_{|u_3| \leq 1}
\bigg(\frac{|u_2|^\ta}{|x_2|^{1+\ta} |x_3|}+\frac{|u_3|^\ta}{|x_2|
|x_3|^{1+\ta}}\Big)\Big (\Big|\frac{4x_2}{x_3}\Big|+\Big|\frac{x_3}{4x_2}\Big|
\Big)^{-\lambda}du_3 du_2\\
&\lesssim \frac{1}{(1+|x_1|)^{1+\lambda} (1+|x_2|)^{1+\lambda} (1+|x_3|)^{1+\lambda
}}.
\end{align*}

Case 6. $|x_1|< 3, |x_2| \geq 3, |x_3|< 3.$ Note that
\begin{align*}
S &= \lim_{\epsilon_1,\epsilon_2,\epsilon_3 \to 0 \atop N_1,N_2,N_3
\to \infty}\int_{\epsilon_1 \leq |u_1|\leq 4}\int_{\epsilon_2 \leq
|x_2-u_2|\leq N_2}\int_{\epsilon_3 \leq
|x_3-u_3|\leq N_3} \mathcal K(u_1,x_2-u_2,x_3-u_3)\\
&\hskip 6cm\times \big(\phi^{(1)}(x_1-u_1)-\phi^{(1)}(x_1)\big) \phi^{(2)}(u_2,u_3)du_3du_2du_1\\
&\qquad+\lim_{\epsilon_1,\epsilon_2,\epsilon_3 \to 0 \atop N_1,N_2,N_3
\to \infty}\int_{\epsilon_1 \leq |u_1|\leq 4}\int_{\epsilon_2 \leq
|x_2-u_2|\leq N_2}\int_{\epsilon_3 \leq
|x_3-u_3|\leq N_3} \mathcal K(u_1,x_2-u_2,x_3-u_3)\\
&\hskip 6cm\times \phi^{(1)}(x_1)\phi^{(2)}(u_2,u_3)du_3du_2du_1.
\end{align*}
By the condition (R) with $\alpha=\beta=\gamma=0$ and the smoothness
condition of $\phi^{(1)}$ on the first integral, the condition
(C2.b) with $\beta=\gamma=0$ for the second integral, and the
dominated convergent theorem,
\begin{align*}
|S|
&\lesssim \int_{|u_1|\le 4}\int_{|u_2|\le 1}\int_{|u_3|\le 1}\frac{1}{|u_1| |x_2|
|x_3-u_3|}\Big(\Big|\frac{u_1x_2}{x_3-u_3}\Big|+\Big|\frac{x_3-u_3}{u_1x_2}\Big|
\Big)^{-\lambda}|u_1|du_3du_2du_1\\
&\quad +\int_{|u_2| \leq 1}\int_{|u_3| \leq 1} \frac{1}{|x_2|
|x_3-u_3|}\Big(\Big|\frac{4x_2}{x_3-u_3}\Big|+\Big|\frac{x_3-u_3}{4x_2}\Big|
\Big)^{-\lambda}du_3du_2\\
&\lesssim \frac{1}{|x_2|^{1+\lambda}}\\
&\lesssim \frac{1}{(1+|x|)^{1+\lambda} (1+|x_2|)^{1+\lambda} (1+|x_3|)^{1+\lambda
}}.
\end{align*}

Case 7. $|x_1|< 3, |x_2| < 3, |x_3|  \geq 3$. The required estimate
follows directly from the condition (R):
\begin{align*}
|S|
&\lesssim \int_{|u_1| \leq 1}\int_{|u_2| \leq 1}\int_{|u_3| \leq 1}
\frac{1}{|x_1-u_1| |x_2-u_2| |x_3|}\Big|\frac{x_3}{(x_1-u_1)(x_2-u_2)}\Big|^{-\lambda}du_3du_2du_1\\
&\lesssim \frac{1}{|x_3|^{1+\lambda}}\\
&\lesssim \frac{1}{(1+|x_1|)^{1+\lambda} (1+|x_2|)^{1+\lambda} (1+|x_3|)^{1+\lambda}}.
\end{align*}

Case 8. $|x_1|< 3, |x_2| < 3, |x_3| < 3$. Inserting
\begin{align*}
&\phi^{(1)}(x_1-u_1)\phi^{(2)}(x_2-u_2,x_3-u_3)\\
&\qquad = [\phi^{(1)}(x_1-u_1)-\phi^{(1)}(x_1)][\phi^{(2)}(x_2-u_2,x_3-u_3)-\phi^{(2)}(x_2,x_3)]\\
&\qquad\quad+ \phi^{(1)}(x_1)[\phi^{(2)}(x_2-u_2,x_3-u_3)-\phi^{(2)}(x_2,x_3)]\\
&\qquad\quad+ [\phi^{(1)}(x_1-u_1)-\phi^{(1)}(x_1)]\phi^{(2)}(x_2,x_3)
+\phi^{(1)}(x_1)\phi^{(2)}(x_2,x_3),
\end{align*}
we write
\begin{align*}
S
&= \lim_{\epsilon_1,\epsilon_2,\epsilon_3 \to 0 \atop N_1,N_2,N_3
\to \infty}\int_{\epsilon_1 \leq |u_1|\leq 4}\int_{\epsilon_2 \leq
|u_2|\leq 4}\int_{\epsilon_3 \leq |u_3|\leq 4}
\mathcal K(u_1,u_2,u_3) \\
&\hskip 2cm\times
\Big\{[\phi^{(1)}(x_1-u_1)-\phi^{(1)}(x_1)][\phi^{(2)}(x_2-u_2,x_3-u_3)-\phi^{(2)}(x_2,x_3)]\\
&\hskip 2.7cm + \phi^{(1)}(x_1)[\phi^{(2)}(x_2-u_2,x_3-u_3)-\phi^{(2)}(x_2,x_3)]\\
&\hskip 2.7cm + [\phi^{(1)}(x_1-u_1)-\phi^{(1)}(x_1)]\phi^{(2)}(x_2,x_3)\\
&\hskip 2.7cm +\phi^{(1)}(x_1)\phi^{(2)}(x_2,x_3)\Big\}du_3du_2du_1
\end{align*}
as four integrals. Using the condition (R) with
$\alpha=\beta=\gamma=0$, the smoothness condition of $\phi^{(1)}$
for the first integral, the cancellation conditions (C2.b), (C2.c),
(C2.a) for the last three integrals, and the dominated
convergent theorem, we obtain
\begin{align*}
|S|
&\lesssim \int_{|u_1| \leq 4}\int_{|u_2| \leq 4}\int_{|u_3| \leq 4}
\frac{1}{|u_1| |u_2|
|u_3|}\Big(\Big|\frac{u_1u_2}{u_3}\Big|+\Big|\frac{u_3}{u_1u_2}\Big|
\Big)^{-\tb} |u_1|(|u_2|+|u_3|)du_3du_2du_1\\
&\quad +\int_{|u_2| \leq 4}\int_{|u_3| \leq 4} \frac{1}{|u_2|
|u_3|}\Big(\Big|\frac{4u_2}{u_3}\Big|+\Big|\frac{u_3}{4u_2}\Big|
\Big)^{-\tb}(|u_2|+|u_3|)du_3du_2 \\
&\quad +\int_{|u_1| \leq 4}
\frac{1}{|u_1|} |u_1|du_1+1 \\
&\lesssim 1\\
&\lesssim \frac{1}{(1+|x_1|)^{1+\lambda}
(1+|x_2|)^{1+\lambda} (1+|x_3|)^{1+\lambda }}.
\end{align*}
The  proof of Lemma \ref{Lem 3.2} is completed.
\end{proof}

Recall that
$\phi_{j,k}(u_1,u_2,u_3)=2^{-2(j+k)}\phi^{(1)}(2^{-j}u_1)\phi^{(2)}(2^{-k}u_2,2^{-(j+k)}u_3)$,
and the assumptions on $\mathcal K$ are invariant with respect to
Zygmund dilation. By Lemma \ref{Lem 3.2}, we %directly
have the following estimate
\begin{equation*}
\big| (\mathcal K \ast \phi_{j,k})(x_1,x_2,x_3) \big| \leq
C\frac{2^{-j}}{(1+2^{-j}
|x_1|)^{1+\lambda}}\frac{2^{-k}}{(1+2^{-k}|x_2|)^{1+\lambda
}}\frac{2^{-(j+k)}}{(1+2^{-(j+k)}|x_3|)^{1+\lambda }}.
\end{equation*}
Now the proof of Proposition \ref{Prop 3.1} follows from the above
estimate with replacing $\phi_{j,k}$ by $\phi_{j,k} \ast
\phi_{j',k'}.$ Note that, by Lemma \ref{almost} given below, $\phi_{j,k} \ast \phi_{j',k'}$ satisfies the same properties as
$\phi_{j\vee j', k\vee k'}$ but with the bound
$C2^{-|j-j'|}2^{-|k-k'|}.$
Thus, the proof of Proposition \ref{Prop 3.1} follows and hence Theorem \ref{Thm 1.3} is proved.

The following lemma is an almost orthogonal estimate.

\begin{lem}\label{almost}
Suppose that $\phi_{j,k}$ is defined as in \eqref{eq 3.6}. Then
\begin{align}\label{eq:almost}
\begin{split}
&|\phi_{j,k}*\phi_{j',k'}(x)|  \\
&\lesssim   2^{-|j-j'|L}2^{-|k-k'|L}\frac{2^{M(j\vee j')}} {(2^{j\vee j'}+\vert
x_1\vert )^{1+M}} \frac{2^{M(k\vee k')}}{2^{j^\star}(2^{k\vee k'}+\vert x_2\vert + 2^{-j^*}\vert x_3\vert
 )^{2+M}}
\end{split}
\end{align}
for any fixed $L,M>0$,
where $x=(x_1,x_2,x_3)\in \Bbb R^3, j^\star=j$
if $k\ge k^\prime$ and
$j^\star=j^\prime$ if $k<k^\prime$.
\end{lem}

\begin{proof}
Let $\phi^{(2)}_{j,k}(x_2,x_3)=2^{-j}2^{-2k}\phi^{(2)}(2^{-k}x_2, 2^{-k-j}x_3)$.
Note that $\phi_{j}^{(1)}(x_1)=2^{-j}\phi^{(1)}(2^{-j}x_1).$ Then $\phi_{j,k}(x_1,x_2,x_3)=\phi_{j}^{(1)}(x_1)\phi^{(2)}_{j,k}(x_2,x_3)$,
and hence \eqref{eq:almost} follows if we prove the following estimates:
\begin{equation}\label{eq:almost1}
|\phi_{j}^{(1)}*\phi_{j'}^{(1)}(x_1)|  \lesssim   2^{-|j-j'|L_1} \frac{2^{M(j\vee j')}} {(2^{j\vee j'}+\vert
x_1\vert )^{1+M}}
\end{equation}
and
\begin{align}\label{eq:almost2}
\begin{split}
|\phi_{j,k}^{(2)}*\phi_{j',k'}^{(2)}(x_2,x_3)|\lesssim   2^{(|j-j'|(L_2+M+2)} 2^{-|k-k'|L_2} \frac{2^{M(k\vee k')}}{2^{j^\star}(2^{k\vee k'}+\vert x_2\vert + 2^{-j^*}\vert x_3\vert
 )^{2+M}}
\end{split}
\end{align}
for any fixed $L_1,L_2,M>0.$

Inequality \eqref{eq:almost1} is the classical almost orthogonality estimate and thus it suffices to show \eqref{eq:almost2}.

By symmetry, we can only consider the case when $k>k^\prime$. Applying the cancellation conditions on $\phi_{j^\prime, k^\prime}^{(2)}$ and the smoothness conditions on $\phi_{j,k}^{(2)},$ we write
\begin{align}\label{almost: 1}
\begin{split}
&
\left|\int_{\mathbb R^2}\phi_{j,k}^{(2)}(x_2-u_2, x_3-u_3)\phi_{j^\prime, k^\prime}^{(2)}(u_2, u_3)du_2 du_3\right|  \\
&\quad=  \left|\int_{\mathbb R^2}[\phi_{j,k}^{(2)}(x_2-u_2, x_3-u_3)-\mathcal P_{L-1}[\phi_{j,k}^{(2)}](x_2, x_3)]\phi_{j^\prime, k^\prime}^{(2)}(u_2, u_3)du_2 du_3\right|\\
&\quad\lesssim  \int_{\mathbb R^2} \Big(\frac{|u_2|}{2^k}+\frac{|u_3|}{2^{j+k}}\Big)^L \frac{2^{-2k-j}}{(1+2^{-k}|\xi_2|+2^{-j-k}|\xi_3|)^{M_1}}\\
&\qquad\times \frac{2^{-2k'-j'}}{(1+2^{-k'}|u_2|+2^{-j'-k'}|u_3|)^{M_2}}du_2 du_3
\end{split}
\end{align}
for some $(\xi_2,\xi_3)$ on the segment joining $(x_2-u_2, x_3-u_3)$ to $(x_2, x_3)$,
where $\mathcal P_{L-1}[\phi_{j,k}^{(2)}](x_2, x_3)$ denotes the Taylor polynomial of order $L-1$ of $\phi_{j,k}^{(2)}$ at $(x_2, x_3)$.

By the triangle inequality,
\begin{gather}
|x_2|\le |\xi_2|+|x_2-\xi_2|\le |\xi_2|+|u_2|\label{eq: 3.81},\\
|x_3|\le  |\xi_3|+|x_3-\xi_3|\le |\xi_3|+|u_3|\label{eq: 3.91}.
\end{gather}
From \eqref{eq: 3.81} and the fact that $k>k'$,
\begin{equation}\label{eq: 3.101}
2^{-k}|x_2|\le 2^{-k}|\xi_2|+2^{-k}|u_2|\le 2^{-k}|\xi_2|+2^{-k'}|u_2|.
\end{equation}
Using \eqref{eq: 3.91} and $k>k'$ again, we get
\begin{align}\label{eq: 3.111}
\begin{split}
2^{-j-k}|x_3|\le & 2^{-j-k}|\xi_3|+2^{-j-k}|u_3|\le 2^{-j-k}|\xi_3|+2^{|j-j'|}2^{-j'-k'}|u_3|\\
\le & 2^{|j-j'|}(2^{-j-k}|\xi_3|+2^{-j'-k'}|u_3|).
\end{split}
\end{align}
Putting \eqref{eq: 3.101} and \eqref{eq: 3.111} together gives
\begin{align*}
1+2^{-k}|x_2|+2^{-j-k}|x_3|
\le & 2^{|j-j'|}(1+2^{-k}|\xi_2|+2^{-k'}|u_2| +2^{-j-k}|\xi_3|+2^{-j'-k'}|u_3|)\\
\le & 2^{|j-j'|}(1+2^{-k}|\xi_2|+2^{-j-k}|\xi_3|)(1+2^{-k'}|u_2|+2^{-j'-k'}|u_3| ).
\end{align*}
This is equivalent to
$$
\frac{1}{1+2^{-k}|\xi_2|+2^{-j-k}|\xi_3|}\le 2^{|j-j'|}\frac{1+2^{-k'}|u_2|+2^{-j'-k'}|u_3| }{1+2^{-k}|x_2|+2^{-j-k}|x_3|}.
$$
We also have
$$
\Big(\frac{|u_2|}{2^k}+\frac{|u_3|}{2^{j+k}}\Big)^L\le
2^{[|j-j'|+(k'-k)]L}
\Big(\frac{|u_2|}{2^{k'}}+\frac{|u_3|}{2^{j'+k'}}\Big)^L.
$$

We insert these estimates to the last integral in \eqref{almost: 1} and use the fact that $M_2>M_1+L+2$ to get
\begin{align*}
&
\left|\int_{\mathbb R^2}\phi_{j,k}^{(2)}(x_2-u_2, x_3-u_3)\phi_{j^\prime, k^\prime}^{(2)}(u_2, u_3)du_2 du_3\right|  \\
&\quad\lesssim 2^{(k'-k)L} 2^{|j-j'|(L+M_1)}
\frac{2^{-2k-j} }{(1+2^{-k}|x_2|+2^{-j-k}|x_3|)^{M_1}}\\
&\qquad \times \int_{\mathbb R^2}
 \frac{2^{-2k'-j'}}{(1+2^{-k'}|u_2|+2^{-j'-k'}|u_3|)^{M_2-M_1-L}}du_2 du_3\\
&\quad\lesssim 2^{(k'-k)L} 2^{|j-j'|(L+M_1)}
\frac{2^{-2k-j} }{(1+2^{-k}|x_2|+2^{-j-k}|x_3|)^{M_1}},
\end{align*}
which gives \eqref{eq:almost2} with $L=L_2, M_1=M+2$.
This concludes the proof of Lemma \ref{almost}.
\end{proof}

We now turn to the proof of Theorem \ref{Thm 1.4}. To prove part
(a), we first show the $L^2$ boundedness of $\mathcal K\ast f.$ This
is similar to the proof of Theorem \ref{Thm 1.1}. We only outline
the proof as follows.

By the Plancherel theorem, the $L^2$ boundedness of $\mathcal K\ast
f$ is equivalent to $|\widehat{\mathcal K}(\chi,\eta,\xi)|$ $\leq
A,$ where $\widehat{\mathcal K}$ is the Fourier transform of
$\mathcal K$ in the sense of distributions and $A$ is the constant
depending only on the constant $C$.

Let $\zeta_1(x_1)$ be a smooth function on $\Bbb R$ with
 $\zeta_1(x_1)=1$ if $|x_1| \leq 8$ and $\zeta_1(x_1)=0$ if $|x_1| \geq 16$, and let $\zeta_2=1-\zeta_1$.
 For simplicity, we denote by $\Kb = \frac{1}{\chi\eta\xi}\K(\frac{x_1}{\chi},\frac{x_2}{\eta},\frac{x_3}{\xi})$.
We write
\begin{align*}
\widehat{\K}(\chi,\eta,\xi)
 = &\iiint\Kb \zeta_2(x_1)e^{-ix_1}e^{-ix_2}e^{-ix_3}dx_1dx_2dx_3 \\
 &+\iiint \Kb \zeta_1(x_1)  e^{-ix_1}e^{-ix_2} e^{-ix_3}dx_1dx_2dx_3\\
:=&I+II.
\end{align*}
To estimate $I,$ we write
\begin{align*}
|I|
&= \12 \bigg|\iiint \Delta_{x_1,\pi} \Big(\Kb
\zeta_2(x_1)\Big)e^{-ix_1}e^{-ix_2}e^{-ix_3}dx_1dx_2dx_3\bigg|\\
&\lesssim \bigg|\iiint \Delta_{x_1,\pi} \Big( \Kb \zeta_2(x_1)
\Big)e^{-ix_1}\zeta_2(x_2)e^{-ix_2}e^{-ix_3 }dx_1dx_2dx_3\bigg| \\
&\quad + \bigg|\iiint \Delta_{x_1,\pi} \Big( \Kb \zeta_2(x_1)
\Big) e^{-ix_1}\zeta_1(x_2)e^{-ix_2}e^{-ix_3 }dx_1dx_2dx_3\bigg|\\
&:= I_1+I_2.
\end{align*}
Note that
\begin{align*}
I_{1}
&\lesssim \bigg|\iiint \Delta_{x_2,\pi} \bigg(\Delta_{x_1,\pi} \Big( \Kb \zeta_2(x_1)
\Big)\zeta_2(x_2)\bigg)e^{-ix_1}e^{-ix_2}e^{-ix_3 }dx_1dx_2dx_3\bigg|\\
&\lesssim   \int_{\Bbb R}\int_{|x_2| \geq 8}\int_{|x_1|\geq
8} \frac{1}{|x_1|^{1+\ta} |x_2|^{1+\ta} |x_3|}\xyzall dx_1 dx_2 dx_3\\
& \lesssim 1.
\end{align*}
For term $I_{2},$ note that
\begin{align*}
|I_{2}|
&\le \big|\iiint \Delta_{x_1,\pi} \big( \Kb
\zeta_2(x_1) \big) e^{-ix_1}\zeta_1(x_2)e^{-ix_2}\zeta_2(x_3)e^{-ix_3
}dx_1dx_2dx_3\big| \\
&\quad + \big|\iiint \Delta_{x_1,\pi} \big( \Kb \zeta_2(x_1) \big)
e^{-ix_1}\zeta_1(x_2)e^{-ix_2}\zeta_1(x_3)e^{-ix_3 }dx_1dx_2dx_3\big|\\
&:= I_{2,1}+ I_{2,2}.
\end{align*}
Thus,
\begin{align*}
I_{2,1}
&= \12 \bigg| \iiint \Delta_{x_3,\pi} \Big(\Delta_{x_1,\pi} \big(\Kb
\zeta_2(x_1)\big)\zeta_2(x_3)\Big)
e^{-ix_1}\zeta_1(x_2)e^{-ix_2}e^{-ix_3 }dx_1dx_2dx_3 \bigg|\\
&\lesssim \int_{|x_3| \geq 8}\int_{|x_2| < 16}\int_{|x_1|\geq 8}
\frac{1}{|x_1|^{1+\ta} |x_2| |x_3|^{1+\ta}}\xyzall dx_1dx_2dx_3\\
&\lesssim 1.
\end{align*}
To estimate $I_{2,2},$ we write
\begin{align*}
I_{2,2}
&=\iiint \Delta_{x_1,\pi} \big( \Kb \zeta_2(x_1) \big)
e^{-ix_1}\zeta_1(x_2)\zeta_1(x_3)\big(e^{-ix_2}e^{-ix_3}-1\big)dx_1dx_2dx_3\\
&\quad + \iiint \Delta_{x_1,\pi} \big( \Kb \zeta_2(x_1) \big)
e^{-ix_1}\zeta_1(x_2)\zeta_1(x_3)dx_1dx_2dx_3.
\end{align*}
Inserting $|e^{-ix_2}e^{-ix_3 }-1|\leq |x_2|+|x_3 |$ into the first integral
together with the condition (R) and using the cancellation condition
(C3.c) for the second integral, we get
\begin{align*}
I_{2,2}
&\lesssim \int_{|x_3 | < 16}\int_{|x_2| < 16}\int_{|x_1|\geq 8}
\frac{1}{|x_1|^{1+\ta} |x_2| |x_3|}\xyzall (|x_2|+|x_3|)dx_2dx_3 \\
&\quad + \int_{|x_1|\geq 8}|x_1|^{-1-\ta}dx_1,
\end{align*}
which is dominated by a constant. Altogether, we obtain
the required bound for term $I.$

Now we estimate term $II.$ We first write
\begin{align*}
II
&= \iiint \Kb \zeta_1(x_1)(e^{-ix_1}-1)e^{-ix_2}e^{-ix_3 }dx_1dx_2dx_3\\
&\quad +\iiint \Kb \zeta_1(x_1)e^{-ix_2} e^{-ix_3 }dx_1dx_2dx_3\\
&:= II_1 +II_2.
\end{align*}
We further write
\begin{align*}
II_1=& \iiint \Kb
\zeta_1(x_1)(e^{-ix_1}-1)\zeta_2(x_2)e^{-ix_2} e^{-ix_3 }dx_1dx_2dx_3 \\
&+\iiint \Kb
 \zeta_1(x_1)(e^{-ix_1}-1)\zeta_1(x_2) e^{-ix_2} e^{-ix_3 }dx_1dx_2dx_3\\
:=& II_{1,1}+ II_{1,2}.
\end{align*}

For term $II_{1,1},$ we have
\begin{align*}
|II_{1,1}|= \12 \bigg| \iiint \Delta_{x_2,\pi} \big(\Kb
\zeta_2(x_2)\big)\zeta_1(x_1)(e^{-ix_1}-1)e^{-ix_2} e^{-ix_3 }dx_1dx_2dx_3 \bigg|.
\end{align*}
Then the required bound follows from the fact that $|e^{-ix_1}-1|\leq
|x_1|$ and the condition (R).

Similarly, we write
\begin{align*}
II_{1,2}
&=\iiint \Kb \zeta_1(x_1)(e^{-ix_1}-1)\zeta_1(x_2)e^{-ix_2}
\zeta_2(x_3)e^{-ix_3 }dx_1dx_2dx_3\\
&\quad + \iiint\Kb \zeta_1(x_1)(e^{-ix_1}-1)\zeta_1(x_2)e^{-ix_2}
\zeta_1(x_3)e^{-ix_3 }dx_1dx_2dx_3 \\
&:=II_{1,2,1}+ II_{1,2,2}.
\end{align*}
Since
\begin{equation*}
|II_{1,2,1}| = \12 \bigg|\iiint \Delta_{x_3,\pi} \big(\Kb \zeta_2(x_3)\big)
\zeta_1(x_1)(e^{-ix_1}-1)\zeta_1(x_2)e^{-ix_2}e^{-ix_3 }dx_1dx_2dx_3 \bigg|.
\end{equation*}
The required bound for $II_{1,2,1}$ is concluded by the fact
that $|e^{-ix_1}-1|\leq |x_1|$ and the condition (R). To estimate term
$II_{1,2,2},$ we write
\begin{align*}
II_{1,2,2}
&=\iiint\Kb
\zeta_1(x_1)(e^{-ix_1}-1)\zeta_1(x_2)\zeta_1(x_3)\big(e^{-ix_2}
e^{-ix_3}-1\big)dx_1dx_2dx_3\\
&\quad +\iiint\Kb \zeta_1(x_1)(e^{-ix_1}-1)\zeta_1(x_2)\zeta_1(x_3)dx_1dx_2dx_3.
\end{align*}
Using the facts that $|e^{-ix_1}-1|\leq |x_1|$ and $|e^{-ix_2}e^{-ix_3}-1|\leq |x_2|+|x_3|$,
the condition (R) for the first integral, and the condition (C3.c) for the second integral,
we obtain the desired bound for $II_{1,2,2}.$

Finally, we estimate term $II_2.$ Denote $II_2=II_{2,1}+II_{2,2},$
where $II_{2,1}$ and $II_{2,2}$ are given by $\iiint \Kb
\zeta_1(x_1)\zeta_2(x_2)e^{-ix_2}$ $e^{-ix_3}dx_1dx_2dx_3$ and $\iiint \Kb
\zeta_1(x_1)\zeta_1(x_2)e^{-ix_2}$ $e^{-ix_3}dx_1dx_2dx_3,$ respectively. Then
\begin{eqnarray*}
|II_{2,1}|= \12 \bigg|\iiint \Delta_{x_2,\pi} \big( \Kb
\zeta_2(x_2)\big)\zeta_1(x_1)e^{-ix_2}e^{-ix_3 }dx_1dx_2dx_3 \bigg| \lesssim 1.
\end{eqnarray*}
For $II_{2,2}$, we insert
\begin{align*}
\zeta_1(x_1)\zeta_1(x_2)e^{-ix_2} e^{-ix_3}
&= \zeta_1(x_1)\zeta_1(x_2)e^{-ix_2}
\zeta_2(x_3)e^{-ix_3}\\
&\quad +\zeta_1(x_1)\zeta_1(x_2) \zeta_1(x_3)\big(e^{-ix_2}e^{-ix_3}-1\big)
+ \zeta_1(x_1)\zeta_1(x_2) \zeta_1(x_3)
\end{align*}
into
\begin{eqnarray*}
\iiint \Kb \zeta_1(x_1)\zeta_1(x_2)e^{-ix_2} e^{-ix_3}dx_1dx_2dx_3
\end{eqnarray*}
and apply condition (C1.b), (C1.b), and (C1.a).
Thus these estimates yield the bound of $II_{2,2}$ and hence the
required bound for term $II.$ The $L^2$ boundedness of $\mathcal
K\ast f$ follows.

Next, to show the $L^p$ boundedness of the operator $\mathcal K\ast
f$, similar to the proof of Theorem \ref{Thm 1.3}, it suffices to prove the
following lemma.

\begin{lem}\label{Lem 3.3}
Suppose that $\phi^{(1)}$ and $\phi^{(2)}$ satisfy the conditions
\eqref{eq 3.1} -- \eqref{eq 3.4} and $\mathcal{K}$ is a distribution
defined on $\Bbb R^3$ satisfying conditions {\rm(R)} and {\rm(C3.a)} --
{\rm(C3.c)}. Then, for $\lambda =\12 \min(\ta,\tb),$
\begin{eqnarray*}
|\mathcal K\ast (\phi^{(1)}\otimes \phi^{(2)})(x_1,x_2,x_3)|\leq
\frac{C_\lambda}{(1+|x_1|)^{1+\lambda} (1+|x_2|)^{1+\lambda}
(1+|x_3|)^{1+\lambda}},
\end{eqnarray*}
where $C_\lambda$ is the constant depending only on $\lambda.$
\end{lem}
\begin{proof} The proof the Lemma \ref{Lem 3.3} is similar
to the proof of Lemma \ref{Lem 3.2}. For simplicity, let $S=\mathcal
K\ast (\phi^{(1)}\otimes \phi^{(2)})(x_1,x_2,x_3)$. We consider the
following eight cases.

Case 1. $|x_1|\geq 3, |x_2| \geq 3, |x_3| \geq 3$. For this case, we use
(\ref{eq 3.4}) to write
\begin{align*}
S
&=\iiint \big[\mathcal K(x_1-u_1,x_2-u_2,x_3-u_3)-\mathcal
K(x_1,x_2-u_2,x_3-u_3)\\
&\qquad\quad -\mathcal K(x_1-u_1,x_2,x_3)+\mathcal K(x_1,x_2,x_3)\big]
\phi^{(1)}(u_1)\phi^{(2)}(u_2,u_3)du_3du_2du_1.
\end{align*}
Note that $ \mathcal K(x_1-u_1,x_2-u_2,x_3-u_3)-\mathcal K(x_1,x_2-u_2,x_3-u_3)-
\big(\mathcal K(x_1-u_1,x_2,x_3)-\mathcal K(x_1,x_2,x_3)\big) = \Delta_{x_2,-u_2}\Delta_{x_1,-u_1}
\K(x_1,x_2,x_3-u_3)+\Delta_{x_3,-u_3} \Delta_{x_1,-u_1} \K (x_1,x_2,x_3).$ Thus, by the condition (R) with
$\alpha=\beta=1, \gamma=0$ and $\alpha=\gamma=1,\beta=0$, respectively,
\begin{align*}
|S|
&\lesssim \int_{|u_1| \leq 1}\int_{|u_2| \leq 1}\int_{|u_3| \leq 1}
\frac{|u_1|^\ta |u_2|^\ta}{|x_1|^{1+\ta} |x_2|^{1+\ta}
|x_3|}\Big(\Big|\frac{x_1x_2}{x_3}\Big|+\Big|
\frac{x_3}{x_1x_2}\Big|\Big)^{-\lambda}du_3du_2du_1\\
&\quad +\int_{|u_1| \leq 1}\int_{|u_2| \leq 1}\int_{|u_3| \leq 1} \frac{|u_1|^\ta
|u_3|^\ta}{|x_1|^{1+\ta } |x_2| |x_3|^{1+\ta
}}\Big(\Big|\frac{x_1x_2}{x_3}\Big|+\Big|
\frac{x_3}{x_1x_2}\Big|\Big)^{-\lambda}du_3du_2du_1\\
&\lesssim \frac{1}{(1+|x_1|)^{1+\lambda} (1+|x_2|)^{1+\lambda}
(1+|x_3|)^{1+\lambda }}.
\end{align*}

Case 2. $|x_1|\geq 3, |x_2| \geq 3, |x_3| < 3$. By the cancellation
condition of $\phi^{(1)},$
\begin{align*}
S
&=\iiint \big[\mathcal K(x_1-u_1,x_2-u_2,x_3-u_3)-\mathcal K(x_1,x_2-u_2,x_3-u_3) \big]\\
&\quad\qquad\times \phi^{(1)}(u_1)\phi^{(2)}(u_2,u_3)du_3du_2du_1.
\end{align*}
Therefore, using the condition (R) with $\alpha=1$ and $\beta=\gamma=0,$ we obtain
\begin{align*}
|S|
&\lesssim \int_{|u_1| \leq 1}\int_{|u_2| \leq 1}\int_{|u_3| \leq 1}
\frac{|u_1|^\ta}{|x_1|^{1+\ta} |x_2| |x_3-u_3|}\Big|\frac{x_1x_2}{x_3-u_3}\Big|^{-\lambda}du_3du_2du_1\\
&\lesssim \frac{1}{(1+|x_1|)^{1+\lambda} (1+|x_2|)^{1+\lambda}
(1+|x_3|)^{1+\lambda }}.
\end{align*}

Case 3. $|x_1|\geq 3, |x_2| < 3, |x_3|  \geq 3.$ The same expression for
$S$ as in Case 2 yields
\begin{align*}
|S|
&\lesssim \int_{|u_1| \leq 1}\int_{|u_2| \leq 1}\int_{|u_3| \leq 1}
\frac{|u_1|^\ta }{|x_1|^{1+\ta} |x_2-u_2|
|x_3|}\Big|\frac{x_3}{x_1(x_2-u_2)}\Big|^{-\lambda}du_3du_2du_1\\
&\lesssim \frac{1}{(1+|x_1|)^{1+\lambda} (1+|x_2|)^{1+\lambda}
(1+|x_3|)^{1+\lambda}}.
\end{align*}

Before handling the other cases, we introduce a  bump function
$\widetilde{\phi}$ on $\Bbb R$, with $\widetilde{\phi}(x_1)=1$ if $|x_1|\leq 1/2$ and
$\widetilde{\phi}(x_1)=0$ if
$|x_1|\geq 1$.\\

Case 4. $|x_1|\geq 3, |x_2| < 3, |x_3| < 3.$ Using the cancellation
condition of $\phi^{(1)}$, we write
\begin{align*}
\mathcal K&\ast (\phi^{(1)}\otimes \phi^{(2)})(x_1,x_2,x_3)\\
=& \iiint\big(\mathcal K(x_1-u_1,u_2,u_3)-\mathcal K(x,u_2,u_3)\big) \phi^{(1)}(u_1)\\
&\qquad\times \big(\phi^{(2)}(x_2-u_2,x_3-u_3)-\phi^{(2)}(x_2,x_3)\big)
  \widetilde{\phi}\Big(\frac{u_2}{10}\Big)\widetilde{\phi}\Big(\frac{u_3}{10}\Big)du_3du_2du_1\\
& +\iiint\big(\mathcal K(x_1-u_1,u_2,u_3)-\mathcal K(x_1,u_2,u_3)\big)\phi^{(1)}(u_1)\\
&\hskip 1.3cm \times \phi^{(2)}(x_2,x_3)
\widetilde{\phi}\Big(\frac{u_2}{10}\Big)\widetilde{\phi}\Big(\frac{u_3}{10}\Big)du_3du_2du_1.
\end{align*}
Hence, by the condition (R) with $\alpha=1,\beta=\gamma=0$ for the first
integral and the cancellation condition (C3.c) with $\alpha=1$ for
the second integral,
\begin{align*}
|S|
&\lesssim \int_{|u_1| \leq 1}\int_{|u_2| \leq 10}\int_{|u_3| \leq 10}
\frac{|u_1|^\ta}{|x_1|^{1+\ta} |u_2|
|u_3|}\Big(\Big|\frac{x_1u_2}{u_3}\Big|+\Big|\frac{u_3}{x_1u_2}\Big|
\Big)^{-\lambda} (|u_2|+|u_3|)du_3du_2du_1\\
&\quad + \int_{|u_1| \leq 1} \frac{|u_1|^\ta}{|x_1|^{1+\ta}} du_1\\
&\lesssim \frac{1}{(1+|x_1|)^{1+\lambda}(1+|x_2|)^{1+\lambda} (1+|x_3|)^{1+\lambda }}.
\end{align*}

Case 5. $|x_1|< 3, |x_2| \geq 3, |x_3| \geq 3$. Similar to Case 4, using
the cancellation condition of $\phi^{(2)},$ we write
\begin{align*}
S
&= \iiint
\big[\mathcal K(u_1,x_2-u_2,x_3-u_3)-\mathcal K(u_1,x_2,x_3)\big] \\
&\qquad\quad\times \big(\phi^{(1)}(x_1-u_1)-\phi^{(1)}(x_1)\big)\phi^{(2)}(u_2,u_3) \widetilde{\phi}\Big(\frac{u_1}{10}\Big)du_3du_2du_1 \\
&\quad+\iiint \big[\mathcal K(u_1,x_2-u_2,x_3-u_3)-\mathcal K(u_1,x_2,x_3)\big]
\phi^{(1)}(x_1)\phi^{(2)}(u_2,u_3) \widetilde{\phi}\Big(\frac{u_1}{10}\Big)du_3du_2du_1.
\end{align*}
Note that $\mathcal K(u_1,x_2-u_2,x_3-u_3)-\mathcal K(u_1,x_2,x_3)=\Delta_{x_2,-u_2} \K
(u_1,x_2,x_3-u_3)\break + \Delta_{x_3,-u_3} \K (u_1,x_2,x_3).$
Using condition (R) on
$\mathcal K$, the smoothness of $\phi^{(1)}$ for the first integral,
the cancellation conditions (C3.b) with $\beta=1, \gamma=0$ and
$\beta=0, \gamma=1$ for the second integral, and
applying the dominated convergence theorem, we obtain
\begin{align*}
|S|
&\lesssim \int_{|u_1|\le 10}\int_{|u_2|\le 1}\int_{|u_3|\le 1}
\bigg(\frac{|u_2|^\ta}{|u_1| |x_2|^{1+\ta } |x_3|}+\frac{|u_3|^\ta}{|u_1| |x_2|
|x_3|^{1+\ta}}\Big)\\
&\hskip 4.3cm\times \Big(\Big|\frac{u_1x_2}{x_3}\Big|+\Big|\frac{x_3}{u_1x_2}\Big|\Big)^{-\lambda}|u_1|du_3du_2du_1\\
&\quad+ \int_{|u_2| \leq 1}\int_{|u_3| \leq 1}
\bigg(\frac{|u_2|^\ta}{|x_2|^{1+\ta} |x_3|}+\frac{|u_3|^\ta}{|x_2|
|x_3|^{1+\ta}}\Big)\Big (\Big|\frac{4x_2}{x_3}\Big|+\Big|\frac{x_3}{4x_2}\Big|
\Big)^{-\lambda}du_3du_2\\
&\lesssim \frac{1}{(1+|x_1|)^{1+\lambda} (1+|x_2|)^{1+\lambda} (1+|x_3|)^{1+\lambda}}.
\end{align*}

Case 6. $|x_1|< 3, |x_2| \geq 3, |x_3|< 3.$ Note that
\begin{align*}
S
&= \iiint \mathcal K(u_1,x_2-u_2,x_3-u_3) \big(\phi^{(1)}(x_1-u_1)-\phi^{(1)}(x_1)\big) \phi^{(2)}(u_2,u_3) \widetilde{\phi}\Big(\frac{u_1}{10}\Big)du_3du_2du_1\\
&\quad+\iiint \mathcal K(u_1,x_2-u_2,x_3-u_3) \phi^{(1)}(x_1)\phi^{(2)}(u_2,u_3)
\widetilde{\phi}\Big(\frac{u_1}{10}\Big)du_3du_2du_1.
\end{align*}
By the condition (R) with $\alpha=\beta=\gamma=0$ and
the smoothness condition of $\phi^{(1)}$ on the first integral,
and the condition (C3.b) with $\beta=\gamma=0$ for the second integral,
\begin{align*}
|S|
&\lesssim \int_{|u_1|\le 10}\int_{|u_2|\le 1}\int_{|u_3|\le 1}
\frac{1}{|u_1| |x_2|
|x_3-u_3|}\Big(\Big|\frac{u_1x_2}{x_3-u_3}\Big|+\Big|\frac{x_3-u_3}{u_1x_2}\Big|
\Big)^{-\lambda}|u_1|du_3du_2du_1\\
&\quad +\int_{|u_2| \leq 1}\int_{|u_3| \leq 1} \frac{1}{|x_2|
|x_3-u_3|}\Big(\Big|\frac{x_2}{x_3-u_3}\Big|+\Big|\frac{x_3-u_3}{x_2}\Big|
\Big)^{-\lambda}du_3du_2\\
&\lesssim \frac{1}{|x_2|^{1+\lambda}}\\
&\lesssim \frac{1}{(1+|x_1|)^{1+\lambda} (1+|x_2|)^{1+\lambda} (1+|x_3|)^{1+\lambda}}.
\end{align*}

Case 7. $|x_1|< 3, |x_2| < 3, |x_3|  \geq 3$. The required estimate
follows directly from the condition (R):
\begin{align*}
|S|
&\lesssim \int_{|u_1| \leq 1}\int_{|u_2| \leq 1}\int_{|u_3| \leq 1}
\frac{1}{|x_1-u_1| |x_2-u_2| |x_3|}\Big|\frac{x_3}{(x_1-u_1)(x_2-u_2)}\Big|^{-\lambda}du_3du_2du_1\\
&\lesssim \frac{1}{|x_3|^{1+\lambda}}\\
&\lesssim \frac{1}{(1+|x_1|)^{1+\lambda} (1+|x_2|)^{1+\lambda} (1+|x_3|)^{1+\lambda}}.
\end{align*}

Case 8. $|x_1|< 3, |x_2| < 3, |x_3| < 3$. Inserting
\begin{align*}
&\phi^{(1)}(x_1-u_1)\phi^{(2)}(x_2-u_2,x_3-u_3)\\
&\qquad= [\phi^{(1)}(x_1-u_1)-\phi^{(1)}(x_1)][\phi^{(2)}(x_2-u_2,x_3-u_3)-\phi^{(2)}(x_2,x_3)]\\
&\quad\qquad+ \phi^{(1)}(x_1)[\phi^{(2)}(x_2-u_2,x_3-u_3)-\phi^{(2)}(x_2,x_3)]\\
&\quad\qquad+ [\phi^{(1)}(x_1-u_1)-\phi^{(1)}(x_1)]\phi^{(2)}(x_2,x_3)
+\phi^{(1)}(x_1)\phi^{(2)}(x_2,x_3),
\end{align*}
we write
\begin{align*}
S
&=\iiint \mathcal K(u_1,u_2,u_3) \times
\Big\{[\phi^{(1)}(x_1-u_1)-\phi^{(1)}(x_1)][\phi^{(2)}(x_2-u_2,x_3-u_3)-\phi^{(2)}(x_2,x_3)]\\
&\quad + \phi^{(1)}(x_1)[\phi^{(2)}(x_2-u_2,x_3-u_3)-\phi^{(2)}(x_2,x_3)]
 + [\phi^{(1)}(x_1-u_1)-\phi^{(1)}(x_1)]\phi^{(2)}(x_2,x_3)\\
&\quad +\phi^{(1)}(x_1)\phi^{(2)}(x_2,x_3)\Big\}\widetilde{\phi}\Big(\frac{u_1}{10}\Big)
 \widetilde{\phi}\Big(\frac{u_2}{10}\Big)\widetilde{\phi}\Big(\frac{u_3}{10}\Big)du_3du_2du_1
\end{align*}
as four integrals. Using the condition (R) with
$\alpha=\beta=\gamma=0$ and the smoothness condition of $\phi^{(1)}$
for the first integral, the cancellation conditions (C3.b), (C3.c)
and (C3.a) for the last three integrals, we obtain
\begin{align*}
|S|
&\lesssim \int_{|u_1| \leq 4}\int_{|u_2| \leq 4}\int_{|u_3| \leq 4}
\frac{1}{|u_1| |u_2|
|u_3|}\Big(\Big|\frac{u_1u_2}{u_3}\Big|+\Big|\frac{u_3}{u_1u_2}\Big|
\Big)^{-\tb} |u_1|(|u_2|+|u_3|)du_3du_2du_1\\
&\quad +\int_{|u_2| \leq 4}\int_{|u_3| \leq 4} \frac{1}{|u_2|
|u_3|}\Big(\Big|\frac{4u_2}{u_3}\Big|+\Big|\frac{u_3}{4u_2}\Big|
\Big)^{-\tb}(|u_2|+|u_3|)du_3du_2\\
&\quad +\int_{|u_1| \leq 4} \frac{1}{|u_1|} |u_1|du_1+1 \\
&\lesssim 1\\
&\lesssim \frac{1}{(1+|x_1|)^{1+\lambda} (1+|x_2|)^{1+\lambda}
(1+|x_3|)^{1+\lambda }}.
\end{align*}
This completes the proof of Lemma \ref{Lem 3.3}.
\end{proof}

 The proof of part (b) of Theorem
\ref{Thm 1.4} follows from part (a). Indeed, the conditions (R) and
(C2.a) -- (C2.c) imply the conditions (C3.a) -- (C3.c). To see this,
inserting
\begin{align*}
\widetilde{\phi}( x_1,x_2,x_3)
&=\big[(\widetilde{\phi}(x_1,x_2,x_3)-\widetilde{\phi}(0,x_2,x_3))-(\widetilde{\phi}(x_1,0,0)-\widetilde{\phi}(0,0,0))\big]\\
&\quad+(\widetilde{\phi}(x_1,0,0)-\widetilde{\phi}(0,0,0))
+(\widetilde{\phi}(0,x_2,x_3)-\widetilde{\phi}(0,x_2,0))\\
&\quad+(\widetilde{\phi}(0,x_2,0)-\widetilde{\phi}(0,0,0))+\widetilde{\phi}(0,0,0)
\end{align*}
into $\iiint\mathcal{K}(x_1,x_2,x_3)\widetilde{\phi}( x_1,  x_2, x_3)dx_1dx_2dx_3,$ we obtain that
\begin{align*}
&\iiint \mathcal{K}(x_1,x_2,x_3)\widetilde{\phi}(R_1 x_1, R_2 x_2, R_1R_2 x_3)dx_1dx_2dx_3\\
&\qquad =\lim_{\epsilon_1,\epsilon_2,\epsilon_3\rightarrow 0}\ir
\mathcal{K}(x_1,x_2,x_3)\widetilde{\phi}(R_1 x_1, R_2 x_2, R_1R_2 x_3)dx_1dx_2dx_3.
\end{align*}
Let $E(\epsilon,R_1,R_2)=\{x\in \Bbb R^3: \epsilon_1 \leq |x_1|\leq
\frac{1}{R_1},\epsilon_2 \leq |x_2|\leq \frac{1}{R_2},\epsilon_3 \leq |x_3|\leq \frac{1}{R_1R_2}\}$.
Then
\begin{align*}
&\Bigg|\ir \mathcal{K}(x_1,x_2,x_3)\widetilde{\phi}(R_1 x_1, R_2 x_2, R_1R_2
x_3)dx_1dx_2dx_3\Bigg|\\
%%%%%%%%%%%%%%%%%%term0
&\qquad\leq\Bigg|\ir \mathcal{K}(x_1,x_2,x_3)\Big\{[\widetilde{\phi}(R_1 x_1, R_2 x_2,
R_1R_2 x_3)-\widetilde{\phi}(0,R_2 x_2,R_1R_2 x_3)] \\
&\hskip 6cm -[\widetilde{\phi}(R_1
x_1,0,0)-\widetilde{\phi}(0,0,0)]\Big\} dx_1dx_2dx_3\Bigg|\\
%%%%%%%%%%%%%%term1
&\qquad\quad+\Bigg|\ir \mathcal{K}(x_1,x_2,x_3)(\widetilde{\phi}(R_1
x_1,0,0)-\widetilde{\phi}(0,0,0))dx_1dx_2dx_3\Bigg|\\
%%%%%%%%%%%%%term2
&\qquad\quad+\Bigg|\ir \mathcal{K}(x_1, x_2,
x_3)(\widetilde{\phi}(0,R_2 x_2,R_1R_2 x_3)-\widetilde{\phi}(0,R_2 x_2,0))dx_1dx_2dx_3\Bigg|\\
%%%%%%%%%%%%%%%%%term3
&\qquad\quad+\Bigg|\ir \mathcal{K}(x_1, x_2,
x_3)(\widetilde{\phi}(0,R_2 x_2,0)-\widetilde{\phi}(0,0,0))dx_1dx_2dx_3\Bigg|\\
%%%%%%%%%%%term4
&\qquad\quad+\Bigg|\ir \mathcal{K}(x_1, x_2,
x_3)\widetilde{\phi}(0,0,0)dx_1dx_2dx_3\Bigg|\\
%%%%%%%%%%%%%%%term5
&\qquad\lesssim\irb \frac{1}{|x_1| |x_2|
|x_3|}\Big(\Big|\frac{x_1x_2}{x_3}\Big|+\Big|\frac{x_3}{x_1x_2}\Big|\Big)^{-\tb}\\
&\hskip 6cm\times|R_1x_1|(|R_2x_2|+|R_1R_2x_3|)dx_1dx_2dx_3\\
%%%%%%%%%%%%%%termb1
&\qquad\quad+\int_{|x_1| \leq \frac{1}{R_1}}\frac{1}{|x_1|}|R_1x_1|dx_1\\
%%%%%%%%%%%%%%termb2
&\qquad\quad+\int_{|x_3| \leq \frac{1}{R_1R_2}}\int_{\Bbb R}\frac{1}{|x_2|
|x_3|}\Big(\Big|\frac{x_2}{R_1x_3}\Big|+\Big|\frac{R_1x_3}{x_2}\Big|\Big)^{-\tb}|R_1R_2x_3|dx_2dx_3\\
%%%%%%%%%%%%%%%%%termb3
&\qquad\quad+\Bigg|\int_{|x_2| \leq \frac{1}{R_2}}\frac{1}{|x_2|}|R_2x_2|dx_2
%%%%%%%%%%%termb4
+1\\
&\qquad \lesssim 1,
\end{align*}
where we apply conditions (R) and (C2.c) for the first and second term,
respectively, (C2.b) for the third and fourth term, and (C2.a) for
the last term above. Hence $\mathcal K$ satisfies (C3.a).

Similarly, for any $0 \le \beta+\gamma\leq 1$, n.b.f. $\widetilde{\phi}$ on
$\Bbb R$ and  $R>0,$ we can write
\begin{align*}
&\bigg|\Dy^\beta \Dz^\gamma \mathcal{K}(x_1,x_2,x_3)\widetilde{\phi}(R x_1)dx_1\bigg|\\
&\qquad=\lim\limits_{\epsilon \rightarrow 0}\bigg|\int_{\epsilon
\leq |x_1| \leq \frac{1}{R}} \Dy^\beta \Dz^\gamma \mathcal{K}(x_1, x_2,
x_3)\widetilde{\phi}(R x_1)dx_1\bigg|
\end{align*}
and
\begin{align*}
&\bigg|\int_{\epsilon \leq |x_1| \leq \frac{1}{R}} \Dy^\beta
\Dz^\gamma \mathcal{K}(x_1,x_2,x_3)\widetilde{\phi}(R x_1)dx_1\bigg|
%%%%%%%%%%%%%%%%%%term0
\\&\qquad\leq\bigg|\int_{\epsilon \leq |x_1| \leq \frac{1}{R}}
\Dy^\beta \Dz^\gamma \mathcal{K}(x_1,x_2,x_3)(\widetilde{\phi}(R x_1)-\widetilde{\phi}(0))dx_1\bigg|\\
%%%%%%%%%%%%%%term1
&\qquad\quad+\bigg|\int_{\epsilon \leq |x_1| \leq \frac{1}{R}} \Dy^\beta
\Dz^\gamma \mathcal{K}(x_1,x_2,x_3)\widetilde{\phi}(0)dx_1\bigg|
%%%%%%%%%%%%%%term2
\\&\qquad\lesssim \int_{ |x_1| \leq \frac{1}{R}} \frac{|h_2|^\tab |h_3|^\tac}{|x_1|
|x_2|^{\tab+1}
|x_3|^{\tac+1}}\Big(\Big|\frac{x_1x_2}{x_3}\Big|+\Big|\frac{x_3}{x_1x_2}\Big|\Big)^{-\tb}|Rx_1|dx_1
%%%%%%%%%%%%%%termb1
 \\&\qquad\quad+\frac{|h_2|^\tab |h_3|^\tac}{|x_2|^{\tab+1}
|x_3|^{\tac+1}}\Big(\Big|\frac{x_2}{Rx_3}\Big|+\Big|\frac{Rx_3}{x_2}\Big|\Big)^{-\tb}
%%%%%%%%%%%termb2
+ \frac{|h_2|^\tab |h_3|^\tac}{|x_2|^{\tab+1}
|x_3|^{\tac+1}}\Big(\Big|\frac{\epsilon x_2}{ x_3}\Big|+\Big|\frac{
x_3}{\epsilon x_2}\Big|\Big)^{-\tb}
\\&\qquad\lesssim \frac{|h_2|^\tab |h_3|^\tac}{|x_2|^{\tab+1}
|x_3|^{\tac+1}}\Big(\Big|\frac{x_2}{Rx_3}\Big|+\Big|\frac{Rx_3}{x_2}\Big|\Big)^{-\tb}
+ \frac{|h_2|^\tab |h_3|^\tac}{|x_2|^{\tab+1}
|x_3|^{\tac+1}}\Big(\Big|\frac{x_2}{\epsilon x_3}\Big|+\Big|\frac{\epsilon
x_3}{x_2}\Big|\Big)^{-\tb}.
\end{align*}
Taking $\epsilon \rightarrow 0$, then (C3.b) is obtained.

Finally we verify (C3.c). For any $0\leq \alpha \leq 1$, n.b.f.
$\widetilde{\phi}$ on $\Bbb R^2$ and  $R_1,R_2>0,$ we write
\begin{eqnarray*}
&&\bigg|\iint\Dx^\alpha \mathcal{K}(x_1,x_2,x_3)\widetilde{\phi}(R_1 x_2,
R_2x_3)dx_2dx_3\bigg|\\
&&\qquad=\lim_{\epsilon_1,\epsilon_2 \rightarrow 0}
\bigg|\int_{\epsilon_2\leq |x_3|\leq
\frac{1}{R_2}}\int_{\epsilon_1\leq |x_2|\leq \frac{1}{R_1}}\Dx^\alpha
\mathcal{K}(x_1,x_2,x_3)  \widetilde{\phi}(R_1 x_2, R_2x_3)dx_2dx_3\bigg|
\end{eqnarray*}
and
\begin{align*}
&\bigg|\int_{\epsilon_2\leq |x_3|\leq
\frac{1}{R_2}}\int_{\epsilon_1\leq |x_2|\leq \frac{1}{R_1}}\Dx^\alpha
\mathcal{K}(x_1,x_2,x_3)\widetilde{\phi}(R_1 x_2, R_2x_3)dx_2dx_3\bigg|\\
%%%%%%%%%%%%%%%%%%term0
&\qquad\leq\bigg|\int_{\epsilon_2\leq |x_3|\leq
\frac{1}{R_2}}\int_{\epsilon_1\leq |x_2|\leq \frac{1}{R_1}}\Dx^\alpha
\mathcal{K}(x_1,x_2,x_3)\big(\widetilde{\phi}(R_1 x_2, R_2x_3)-\widetilde{\phi}(0, 0)\big)dx_2dx_3\bigg|\\
%%%%%%%%%%%%%%term1
&\qquad\quad+\bigg|\int_{\epsilon_2\leq |x_3|\leq
\frac{1}{R_2}}\int_{\epsilon_1\leq |x_2|\leq \frac{1}{R_1}}\Dx^\alpha
\mathcal{K}(x_1,x_2,x_3)\widetilde{\phi}(0, 0)dx_2dx_3\bigg|\\
%%%%%%%%%%%%%%term2
&\qquad\lesssim\int_{ |x_3|\leq
\frac{1}{R_2}}\int_{ |x_2|\leq \frac{1}{R_1}}
\frac{|h_1|^\taa}{|x_1|^{\taa+1} |x_2|
|x_3|}\Big(\Big|\frac{x_1x_2}{x_3}\Big|+\Big|\frac{x_3}{x_1x_2}\Big|\Big)^{-\tb}(|R_1x_2|+|R_2x_3|)dx_2dx_3\\
%%%%%%%%%%%%%%termb1
&\qquad\quad+\frac{|h_1|^\taa}{|x_1|^{\taa+1} }
%%%%%%%%%%%termb2
\\&\qquad\lesssim \frac{|h_1|^\taa}{|x_1|^{\taa+1} }.
\end{align*}
Thus (C3.c) is obtained. This completes the proof of part (b), and
hence Theorem \ref{Thm 1.4} is concluded.

%%%%%%%%%%%%%%%%%%%%%%%%%%%%%%%%%%%%%%%%%%%%%%%%%%%%%%%%%%%%%%%%%%%%%%%%%%%%%%%%%%%%%%%%%
%%%%%%%%%%%%%%%%%%%%%%%%%%%%%%%%%%%%%%%%%%%%%%%%%%%%%%%%%%%%%%%%%%%%%%%%%%%%%%%%%%%%%%%%%
\section{Examples and applications}

As mentioned in section 1, the original motivation for this paper is
to introduce a class of singular integral operators which cover those studied by
Ricci and Stein in \cite{RS}. Now in this section we show that a special
class of singular integrals studied by Ricci and Stein \cite{RS}
belongs to our class of singular integrals. Indeed, for
$(x_1,x_2,x_3)\in \Bbb R^3,$ it was proved in \cite{RS} that
$\mathcal{K}(x_1, x_2, x_3)=\sum_{j,k}2^{2j+2k}
\phi^{(1)}(2^jx_1)\phi^{(2)}(2^k x_2,2^{j+k}x_3)$, where
$\phi^{(1)}$ and $\phi^{(2)}$ are defined as in (\ref{eq 3.1}), is a
distribution kernel on $\Bbb R^3.$ The following result shows that
this kernel satisfies the regularity condition (R) and the
cancellation conditions (C2.a) -- (C2.c).

\begin{thm}\label{Thm 4.1}
Suppose that $\phi^{(1)}$ and $\phi^{(2)}$ are defined as in \eqref{eq 3.1} and
$$\mathcal{K}(x_1, x_2, x_3)=\sum_{j,k}2^{2j+2k}\phi^{(1)}(2^jx_1) \phi^{(2)}(2^kx_2,2^{j+k}x_3).$$
Then
\begin{equation}\label{eq 4.1}
|\partial_{x_1}^\alpha\partial_{x_2}^\beta\partial_{x_3}^\gamma
\mathcal{K}(x_1, x_2, x_3)| \leq
\frac{C_{\alpha,\beta,\gamma,\tb}}{|x_1|^{\alpha+1} |x_2|^{\beta+1}
|x_3|^{\gamma+1}}
\Big(\Big|\frac{x_1x_2}{x_3}\Big|+\Big|\frac{x_3}{x_1x_2}\Big|\Big)^{-\tb}
\end{equation}
for all $\alpha,\beta,\gamma\geq 0$ and $0<\tb<1;$
\begin{equation}\label{eq 4.2}
\bigg|\int_{\delta_1 \leq |x_1|\leq r_1}\int_{\delta_2 \leq
|x_2|\leq r_2}\int_{\delta_3 \leq |x_3|\leq r_3} \mathcal{K}(x_1,
x_2, x_3)dx_1dx_2dx_3\bigg| \leq C
\end{equation}
uniformly for all $\delta_1,\delta_2,\delta_3,r_1,r_2,r_3>0;$
\begin{align}\label{eq 4.3}
&\bigg|\int_{\delta \leq|x_1|\leq
r}\partial_{x_2}^\beta\partial_{x_3}^\gamma
\mathcal{K}(x_1, x_2, x_3)dx_1\bigg|\\
&\qquad\leq \frac{C_{\beta,\gamma,\tb}}{|x_1|^{\beta+1}
|x_3|^{1+\gamma}}\Bigg(\frac{1}{\big(|\frac{rx_2}{x_3}|+|\frac{x_3}{rx_2}|
\big)^{\tb}}+\frac{1}{\big(|\frac{\delta
x_2}{x_3}|+|\frac{x_3}{\delta x_2}| \big)^{\tb}}\Bigg),\nonumber
\end{align}
for all $\delta, r>0, \beta, \gamma\geq 0$ and $0<\tb<1;$
\begin{equation}\label{eq 4.4}
\bigg|\int_{\delta_1 \leq|x_2|\leq r_1}\int_{\delta_2 \leq|x_3|\leq
r_2}\partial_{x_1}^\alpha \mathcal{K}(x_1, x_2, x_3)dx_2 dx_3 \bigg|
\leq \frac{C_{\alpha}}{|x_1|^{\alpha+1}}
\end{equation}
uniformly for all $\delta_1,\delta_2,r_1,r_2>0$ and $\alpha\geq 0.$
\end{thm}

To show Theorem \ref{Thm 4.1}, we need the following simple lemmas.

\begin{lem}\label{Lem 4.2}
Suppose $a,b,c>0$  with $b> a$ and $r_1,r_2,r_3>0$. Then, for all
$0<\varepsilon<1,$
$$\sum_j 2^{ja}\frac{1}{(1+2^jr_1)^b}\frac{1}{(r_2+2^jr_3)^c}\leq C_\varepsilon
r_1^{-a}r_2^{-c}\Big(1+\frac{r_3}{r_1r_2}\Big)^{-(a\wedge c)
+1-\varepsilon}.$$
\end{lem}

\begin{proof}
We first write
\begin{align*}
&\sum_j 2^{ja}\frac{1}{(1+2^jr_1)^b}\frac{1}{(r_2+2^jr_3)^c}
= \sum_j 2^{ja}\frac{1}{(1+2^jr_1)^b}\frac{1}{\big(1+2^j\frac{r_3}{r_2}\big)^c}\frac{1}{r_2^c}\\
&\qquad \lesssim  \sum_{j:2^j>r_1^{-1},2^j>r_2r_3^{-1}}
2^{ja}\frac{1}{(2^jr_1)^b}\frac{1}{\big(2^j\frac{r_3}{r_2}\big)^c}\frac{1}{r_2^c}
+\sum_{j:2^j>r_1^{-1},2^j\leq r_2r_3^{-1}}2^{ja}\frac{1}{(2^jr_1)^b}\frac{1}{r_2^c}\\
&\qquad\qquad +\sum_{j:2^j\leq r_1^{-1},2^j>r_2r_3^{-1}}
2^{ja}\frac{1}{\big(2^j\frac{r_3}{r_2}\big)^c}\frac{1}{r_2^c}
+\sum_{j:2^j\leq r_1^{-1},2^j\leq r_2r_3^{-1}} 2^{ja}\frac{1}{r_2^c}\\
&\qquad := I+ II+ III + IV.
\end{align*}
For term $I,$ we observe that
\begin{eqnarray*}
I\lesssim\Big(1+\frac{r_3}{r_1r_2}\Big)^{a-b-c}r_1^{-a}r_2^{-c}\Big(\frac{r_3}{r_1r_2}\Big)^{b-a}
\lesssim r_1^{-a}r_2^{-c}\Big(1+\frac{r_3}{r_1r_2}\Big)^{-c}.
\end{eqnarray*}
For $II$, since $r_3<r_1r_2$,
\begin{eqnarray*}
II \lesssim r_1^{-a}r_2^{-c} \lesssim
r_1^{-a}r_2^{-c}\Big(1+\frac{r_3}{r_1r_2}\Big)^{-c}.
\end{eqnarray*}
For $III,$ note that $r_3>r_1r_2$. We consider three cases. In the
first case where $a>c$, we obtain
\begin{eqnarray*}
III \lesssim r_1^{-a}r_2^{-c}\Big(\frac{r_3}{r_1r_2}\Big)^{-c}
\lesssim r_1^{-a}r_2^{-c}\Big(1+\frac{r_3}{r_1r_2}\Big)^{-c}.
\end{eqnarray*}
If $a<c$, then
\begin{eqnarray*}
III \lesssim r_1^{-a}r_2^{-c}\Big(\frac{r_3}{r_1r_2}\Big)^{-a}
\lesssim r_1^{-a}r_2^{-c}\Big(1+\frac{r_3}{r_1r_2}\Big)^{-a}.
\end{eqnarray*}
When $a=c$, we have
\begin{eqnarray*}
III \lesssim r_3^{-c}\log\Big(\frac{r_3}{r_1r_2}\Big)\lesssim
r_1^{-a}r_2^{-c}\Big(1+\frac{r_3}{r_1r_2}\Big)^{-a+1-\tb}.
\end{eqnarray*}
Finally, for term $IV,$ we have
\begin{eqnarray*}
IV \lesssim \Big(\frac{r_1r_2}{r_3+r_1r_2}\Big)^ar_1^{-a}r_2^{-c}=
r_1^{-a}r_2^{-c}\Big(1+\frac{r_3}{r_1r_2}\Big)^{-a}.
\end{eqnarray*}
These estimates yield the required bound and Lemma \ref{Lem 4.2} is
proved.
\end{proof}

\begin{lem}\label{Lem 4.3}
For any $N>0$, $r>0$ and $k \in \mathbb{Z},$ we have
%$\phi \in C_0^\infty(\Bbb R^n)$. %and $\int_{\Bbb R^n}\phi(x) dx=0$.
$$
\int_{\{x_1\in \Bbb R:|x_1|\leq r\}} \frac{1}{(1+2^k|x_1|)^N} dx_1 \lesssim \frac{r}{1+2^kr}
$$
and
$$
\int_{\{x_1\in \Bbb R: |x_1|> r\}} \frac{1}{(1+2^k|x_1|)^N} dx_1 \lesssim
\frac{2^{-k}}{(1+2^kr)^{N-1}}.
$$
\end{lem}
\begin{proof}
We consider two cases. For the case
$r \leq 2^{-k}$, we clearly have $\int_{|x_1|\leq r}
\frac{1}{(1+2^k|x_1|)^N} dx_1 \lesssim r$. The second inequality follows from
\begin{equation}\label{eq:2.14a}
\int_{\Bbb R} \frac{1}{(1+2^k|x_1|)^N} dx_1 = 2^{-k}\int_{\Bbb R}
\frac{1}{(1+|x_1|)^N} dx_1\lesssim 2^{-k}.
\end{equation}
If $r > 2^{-k}$, then the first inequality follows again from \eqref{eq:2.14a}
while the second follows from
\begin{eqnarray*}
\int_{|x_1>  r} \frac{1}{(1+2^k|x_1|)^N} dx_1 \leq \int_{|x_1|> r}
\frac{1}{(2^k|x_1|)^N} dx_1 \lesssim \frac{2^{-k}}{(1+2^kr)^{N-1}}.
\end{eqnarray*}
The proof of Lemma \ref{Lem 4.3} is finished.
\end{proof}

We now return to show Theorem \ref{Thm 4.1}.
\vskip 0.5cm

\begin{proof}[Proof of Theorem \ref{Thm 4.1}]
 We prove the regularity estimate (\ref{eq 4.1}) first. By the definition of
 $\mathcal K$ and the conditions on $\phi^{(1)}$ and $\phi^{(2)},$
 we have
\begin{align*}
|\partial_{x_1}^\alpha\partial_{x_2}^\beta\partial_{x_3}^\gamma \mathcal K(x_1, x_2, x_3)|
\lesssim \sum_{j,k}\frac{2^{2j+2k+j(\alpha+\gamma)+k(\beta+\gamma)}}
{(1+2^j|x_1|)^{3+\alpha+\gamma}(1+2^k|x_2|+2^{j+k}|x_3|)^3}.
\end{align*}
Note that
\begin{align*}\label{eq 4.5}
\sum_{k}
\frac{2^{2k+k(\beta+\gamma)}}{(1+2^k|x_2|+2^{j+k}|x_3|)^3}
&= \sum_{k:2^k \leq
(|x_2|+2^j|x_3|)^{-1}}\frac{2^{2k+k(\beta+\gamma)}}{(1+2^k|x_2|+2^{j+k}|x_3|)^3}\\ \nonumber
&\qquad + \sum_{k:2^k >
(|x_2|+2^j|x_3|)^{-1}}\frac{2^{2k+k(\beta+\gamma)}}{(1+2^k|x_2|+2^{j+k}|x_3|)^3}\\ %\nonumber
&\lesssim \frac{1}{(|x_2|+2^j|x_3|)^{2+\beta+\gamma}}.
\end{align*}
Inserting this estimate into the above inequality, we obtain
\begin{align*}
|\partial_{x_1}^\alpha\partial_{x_2}^\beta\partial_{x_3}^\gamma \mathcal K(x_1, x_2, x_3)|
&\lesssim \sum_{j}
\frac{2^{2j+j(\alpha+\gamma)}}{(1+2^j|x_1|)^{3+\alpha+\gamma} (|x_2|+2^j|x_3|)^{2+\beta+\gamma}}\\
&\lesssim \frac{1}{|x_1|^{\alpha+\gamma+2} |x_2|^{\beta+\gamma+2}
\big(1+|\frac{x_3}{x_1x_2}| \big)^{(\alpha \wedge
\beta)+\gamma+1+\tb}},
\end{align*}
where we apply Lemma \ref{Lem 4.2} with $a=2+\alpha+\gamma,
b=3+\alpha+\gamma, c=2+\beta+\gamma, r_1=|x_1|, r_2=|x_2|$ and
$r_3=|x_3|$ in the last inequality. This implies the required
estimate.

We now show the cancellation conditions (\ref{eq 4.2}) -- (\ref{eq 4.4}).
To verify (\ref{eq 4.3}), we observe that
\begin{eqnarray*}
&&\bigg|\int_{\delta \leq |x_1|\leq r}\partial_{x_2}^\beta\partial_{x_3}^\gamma \mathcal K(x_1, x_2, x_3)dx_1\bigg|\\
&&\qquad\lesssim \sum_{j,k}2^{2j+2k+k\beta +(j+k)\gamma}\bigg|\int_{\delta
\leq |x_1|\leq r}\phi^{(1)}(2^jx_1)
(\partial_{x_2}^\beta\partial_{x_3}^\gamma\phi^{(2)})(2^k
x_2,2^{j+k}x_3)dx_1\bigg|.
\end{eqnarray*}
Note that, for all $N\geq 2,$
$$\bigg|\int_{|x_1|\leq r}\phi^{(1)}(2^jx_1)dx_1\bigg| \lesssim r,$$
and, by the vanishing condition of $\phi^{(1)},$
\begin{align*}
\bigg|\int_{|x_1|\leq r}\phi^{(1)}(2^jx_1)dx_1\bigg|
&=\bigg|\int_{|x_1|> r}\phi^{(1)}(2^jx_1)dx_1\bigg|\\
&\leq C_N\bigg|\int_{|x_1|> r}\frac{1}{(2^j|x_1|)^N}dx_1\bigg| \leq C_N 2^{-jN}r^{1-N}.
\end{align*}
Therefore,
$$\bigg|\int_{|x_1|\leq r}\phi^{(1)}(2^jx_1)dx_1\bigg| \leq C_N\frac{r}{(1+2^jr)^N}, $$
which implies
\begin{align*}
\bigg|\int_{\delta \leq |x_1|\leq
r}\partial_{x_2}^\beta\partial_{x_3}^\gamma \mathcal K(x_1, x_2, x_3)dx_1\bigg|
&\lesssim \sum_{j,k}
\frac{2^{2j+2k+j\gamma+k(\beta+\gamma)} r}{(1+2^jr)^{3+\gamma} (1+2^k|x_2|+2^{j+k}|x_3|)^3} \\
&\quad +\sum_{j,k}\frac{2^{2j+2k+j\gamma+k(\beta+\gamma)}
\delta}{(1+2^j\delta)^{3+\gamma} (1+2^k|x_2|+2^{j+k}|x_3|)^3}.
\end{align*}
Summing over $k$ first yields that the two summations above are
dominated by
$$\sum_{j}\big(\frac{r}{(1+2^jr)^{3+\gamma}}+\frac{\delta}{(1+2^j\delta)^{3+\gamma}}\big)
\frac{2^{2j+j\gamma}}{(|x_2|+2^j|x_3|)^{2+\beta+\gamma}}.$$
Applying Lemma \ref{Lem 4.2} with $a=2+\gamma, b=3+\gamma,
c=2+\gamma+\beta, r_1=r$ or $\delta$, $r_2=|x_2|, r_3=|x_3|$, we obtain
\begin{align*}
&\bigg|\int_{\delta \leq |x_1|\leq r}\partial_{x_2}^\beta\partial_{x_3}^\gamma \mathcal K(x_1, x_2, x_3)dx_1\bigg|\\
&\qquad\lesssim \frac{1}{r^{\gamma+1} |x_2|^{\beta+\gamma+2}
\big(1+|\frac{x_3}{rx_2}|
\big)^{\gamma+1+\tb}}+\frac{1}{\delta^{\gamma+1}
|x_2|^{\beta+\gamma+2} \big(1+|\frac{x_3}{\delta x_2}|
\big)^{\gamma+1+\tb}},
\end{align*}
which implies the desired cancellation condition (\ref{eq 4.3}).

To show the cancellation condition (\ref{eq 4.4}), we start with
\begin{align*}
&\bigg|\int_{\delta_1 \leq |x_2|\leq r_1}\int_{\delta_2 \leq
|x_3|\leq
r_2}\partial_{x_1}^\alpha \mathcal K(x_1, x_2, x_3)dx_2 dx_3 \bigg| \\
&\qquad \lesssim \sum_{j,k}2^{2j+2k+j\alpha}\bigg|\int_{|x_2|\leq
r_1}\int_{|x_3|\leq
r_2}(\partial_{x_1}^\alpha\phi^{(1)})(2^jx_1)\phi^{(2)}(2^k x_2,2^{j+k}x_3)dx_2 dx_3\bigg|\\
&\qquad\quad+  \sum_{j,k}2^{2j+2k+j\alpha}\bigg|\int_{|x_2|\leq
\delta_1}\int_{|x_3|\leq
\delta_2}(\partial_{x_1}^\alpha\phi^{(1)})(2^jx_1)\phi^{(2)}(2^k
x_2,2^{j+k}x_3)dx_2 dx_3 \bigg|.
\end{align*}
By the vanishing condition of $\phi^{(2)}$,
\begin{align*}
\int_{|x_2|\leq r_1}\int_{|x_3|\leq r_2}\phi^{(2)}(2^k x_2,2^{j+k}x_3)dx_2dx_3
&=\int_{|x_2|> r_1}\int_{|x_3|\leq r_2}\phi^{(2)}(2^k x_2,2^{j+k}x_3)dx_2dx_3\\
&\quad +\int_{|x_2|\leq r_1}\int_{|x_3|> r_2}\phi^{(2)}(2^kx_2,2^{j+k}x_3)dx_2dx_3\\
&\quad +\int_{|x_2|> r_1}\int_{|x_3|> r_2}\phi^{(2)}(2^kx_2,2^{j+k}x_3)dx_2dx_3.
\end{align*}
Applying the size condition of $\phi^{(2)}$ and Lemma \ref{Lem 4.3},
we obtain
\begin{align*}
&\bigg|\int_{|x_2|\leq r_1}\int_{|x_3|\leq r_2}\phi^{(2)}(2^k x_2,2^{j+k}x_3)dx_2dx_3\bigg|\\
&\lesssim \frac{2^{-k}}{(1+2^kr_1)^{3}}\frac{r_2}{1+2^{j+k}r_2}+
\frac{r_1}{1+2^{k}r_1}\frac{2^{-j-k}}{(1+2^{j+k}r_2)^{3}}
+\frac{2^{-k}}{(1+2^kr_1)^{3}}\frac{2^{-j-k}}{(1+2^{j+k}r_2)^{3}}.
\end{align*}
On other hand, the size condition on $\phi^{(2)}$ yields
\begin{eqnarray*}
\bigg|\int_{|x_2|\leq r_1}\int_{|x_3|\leq r_2}\phi^{(2)}(2^k
x_2,2^{j+k}x_3)dx_2d x_3\bigg| \leq
\frac{r_1}{1+2^kr_1}\frac{r_2}{1+2^{j+k}r_2}.
\end{eqnarray*}
Therefore,
\begin{align*}
&\sum_{j,k}2^{2j+2k+j\alpha}\Big| (\partial_{x_1}
\phi^{(1)})(2^jx_1)\int_{|x_2|\leq r_1}\int_{|x_3|\leq r_2}
\phi^{(2)}(2^kx_2,2^{j+k}x_3)(x_2,x_3)dx_2dx_3\Big|\\
&\qquad \lesssim
\sum_{j,k}\frac{2^{2j+2k+j\alpha}}{(1+2^j|x_1|)^{3+\alpha}}
\min\bigg\{\bigg(\frac{2^{-k}}{(1+2^kr_1)^3}\frac{r_2}{1+2^{j+k}r_2}+\frac{r_1}{1+2^{k}r_1}\frac{2^{-j-k}}{(1+2^{j+k}r_2)^{3}}\\
&\hskip 5.8cm+ \frac{2^{-k}}{(1+2^kr_1)^{3}}\frac{2^{-j-k}}{(1+2^{j+k}r_2)^{3}}\bigg),
\frac{r_1}{1+2^kr_1}\frac{r_2}{1+2^{j+k}r_2}\bigg\}.
\end{align*}
Summing over $k$ first and considering the four cases: (i)
$2^{k}\leq r_1^{-1}$ and $2^k \leq 2^{-j}r_2^{-1}$; (ii) $2^{k}\leq
r_1^{-1}$ and $2^k > 2^{-j}r_2^{-1}$; (iii) $2^{k}> r_1^{-1}$ and
$2^k \leq 2^{-j}r_2^{-1}$; (iv) $2^{k}> r_1^{-1}$ and $2^k >
2^{-j}r_2^{-1}$, we obtain that the last summation above is
dominated by
\begin{eqnarray*}
 \sum_{j}\frac{2^{2j+j\alpha}}{(1+2^j|x_1|)^{3+\alpha}}2^{-j},
\end{eqnarray*}
which yields the cancellation condition (\ref{eq 4.4}).

Finally the cancellation (\ref{eq 4.2}) follows directly from the
following estimates.
\begin{align*}
&\bigg|\int_{\delta_1 \leq |x_1|\leq r_1}\int_{\delta_2 \leq |x_2|\leq
r_2}\int_{\delta_3 \leq |x_3|\leq r_3} \mathcal K(x_1, x_2,
x_3)dx_1dx_2dx_3\bigg|\\
&\qquad \lesssim \bigg|\sum_{j,k}2^{2j+2k}\int_{\delta_1\leq |x_1|\leq
r_1} \phi^{(1)}(2^jx_1)dx_1 \\
&\hskip 3cm\times \int_{\delta_2\leq |x_2|\leq
r_2}\int_{\delta_3\leq |x_3|\leq r_3}
\phi^{(2)}(2^kx_2,2^{j+k}x_3)(x_2,x_3)dx_2 dx_3\bigg|\\
&\qquad\lesssim \sum_{j}2^{2j}\bigg(\frac{r_1}{(1+2^jr_1)^3}2^{-j}
+\frac{\delta_1}{(1+2^j\delta_1)^3}2^{-j}\bigg)\\
&\qquad \lesssim 1.
\end{align*}
The proof of Theorem \ref{Thm 4.1} is complete.
\end{proof}

As mentioned in section 1, a special class of singular integral
operators $T_{\frak z}$ considered by Ricci and Stein \cite{RS} is of the
form $T_{\frak z}f=f* \mathcal{K}$, where
$$\mathcal{K}(x_1, x_2, x_3) =\sum _{k,j\in {\Bbb Z}}2^{2(k+j)}
  \phi\Big(2^jx_1,2^k x_2,2^{j+k}x_3\Big)$$
and the function $\phi$ is supported in a unit cube in $\Bbb R^3$
and satisfies a certain amount of uniform smoothness with cancellation conditions
\begin{equation*}
\int _{{\Bbb R}^2}\phi(x_1, x_2, x_3)dx_1dx_2=\int _{{\Bbb R}^2} \phi
(x_1, x_2, x_3)dx_2dx_3=\int _{{\Bbb R}^2}\phi (x_1, x_2, x_3)dx_3dx_1=0.
\end{equation*}
Fefferman and Pipher \cite{FP} showed that the above cancellation conditions are necessary for the $L^2$ boundedness for singular integral $T_{\frak z}$. Moreover, if $\phi$ satisfies the
above cancellation conditions, then $\phi$ can be decomposed by
$\phi=\phi_1+\phi_2,$ where $\phi_1$ and $\phi_2$ have the
following cancellation conditions
\begin{equation*}
\int _{\Bbb R}\phi_1(x_1, x_2, x_3)dx_1=\int _{{\Bbb R}^2}
\phi_1(x_1, x_2, x_3)dx_2dx_3=0
\end{equation*}
and
\begin{equation*}
\int _{\Bbb R}\phi_2(x_1, x_2, x_3)dx_2=\int _{{\Bbb R}^2}
\phi_2(x_1, x_2, x_3)dx_1dx_3=0.
\end{equation*}
This means that the operator $T_{\frak z}$ studied by Ricci and
Stein can be decomposed as $T_{\frak z}=T^1_{\frak
z}+T^2_{\frak z},$ where the kernels of $T^1_{\frak z}$ and
$T^2_{\frak z}$ are given, respectively, by
$$\mathcal K_1(x_1, x_2, x_3)=\sum _{k,j\in {\Bbb Z}}2^{2(k+j)}\phi_1\Big(2^jx_1,2^kx_2,2^{j+k}x_3
\Big)$$ and
$$\mathcal K_2(x_1, x_2, x_3)=\sum _{k,j\in {\Bbb Z}}2^{2(k+j)}\phi_2\Big(2^jx_1,2^kx_2,2^{j+k}x_3
\Big).$$

Theorem \ref{Thm 4.1} shows that the kernel $\mathcal K_1$ satisfies
the regularity (R) and cancellation conditions (C2.a) -- (C2.c)
while the kernel $\mathcal K_2$ satisfies the regularity (R) and
cancellation conditions (C$2^\prime.$a) -- (C$2^\prime.$c).
Therefore, these operators $\mathcal K_1$ and $\mathcal K_2$ belong to our class.

\begin{rem}
Actually, based on the proof of  Theorem \ref{Thm 4.1}, we note
that the kernel
$$\mathcal{K}(x_1, x_2, x_3)=\sum_{j,k\in \Bbb Z} 2^{2j+2k}\phi^{(1)}(2^jx_1) \phi^{(2)}(2^kx_2,2^{j+k}x_3)$$
as in  Theorem \ref{Thm 4.1} satisfies the following stronger conditions
\begin{itemize}\baselineskip=35pt
\item[(i)] $\displaystyle |\partial_{x_1}^\alpha\partial_{x_2}^\beta\partial_{x_3}^\gamma
\mathcal{K}(x_1, x_2, x_3)| \leq
\frac{C_{\alpha,\beta,\gamma,\tb}}{|x_1|^{\alpha+\gamma+2}
|x_2|^{\beta+\gamma+2} \big(1+|\frac{x_3}{x_1x_2}| \big)^{(\alpha
\wedge \beta)+\gamma+1+\tb}}$;
\item[(ii)] $\displaystyle \bigg|\int_{\delta_1 \leq |x_1|\leq r_1}\int_{\delta_2 \leq
|x_2|\leq r_2}\int_{\delta_3 \leq |x_3|\leq r_3} \mathcal{K}(x_1,
x_2, x_3)dx_1dx_2dx_3\bigg| \leq C$
\item[] \hskip 5cm uniformly for all $\delta_1,\delta_2,\delta_3,r_1,r_2,r_3>0$;
\item[(iii)] $\displaystyle \bigg|\int_{\delta \leq|x_1|\leq r}\partial_{x_2}^\beta\partial_{x_3}^\gamma
\mathcal{K}(x_1, x_2, x_3)dx_1\bigg|$
\item[] $\displaystyle \leq C_{\beta,\gamma,\tb}\bigg(\frac{1}{r^{\gamma+1}
|x_2|^{\beta+\gamma+2} \big(1+|\frac{x_3}{rx_2}|
\big)^{\gamma+1+\tb}}+\frac{1}{\delta^{\gamma+1}
|x_2|^{\beta+\gamma+2} \big(1+|\frac{x_3}{\delta x_2}|
\big)^{\gamma+1+\tb}}\bigg)$
\item[] \hskip 5cm for all $\delta, r>0, \beta, \gamma\geq 0$ and $0<\tb<1$;
\item[(iv)] $\displaystyle \bigg|\int_{\delta_1 \leq|x_2|\leq r_1}\int_{\delta_2 \leq|x_3|\leq
r_2}\partial_{x_1}^\alpha \mathcal{K}(x_1, x_2, x_3)dx_2 dx_3 \bigg|
\leq \frac{C_{\alpha}}{|x_1|^{\alpha+1}}$
\item[] \hskip 5cm uniformly for all $\delta_1,\delta_2,r_1,r_2>0$ and $\alpha\geq 0$.
\end{itemize}
\end{rem}

In \cite{NW}, Nagel and Wainger considered the $L^2$ boundedness of certain singular integral
operators on $\Bbb R^n$ whose kernels have  appropriate
homogeneities with respect to a multi-parameter group of dilations,
generated by a finite number of diagonal matrices. In particular,
they considered the following two-parameter dilation group
\begin{equation}\label{eq 2.19}
\delta (s, t)(x_1, x_2, x_3)=(sx_1, tx_2, s^\alpha t^\beta x_3)
\end{equation}
acting on $\Bbb R^3$ for $s, t, \alpha,\beta >0.$ They defined a
singular kernel $\mathcal K$ by
$$\mathcal K(x_1, x_2, x_3) = \text{sgn}(x_1 x_2)
  \bigg\{\frac{|x_1|^{\alpha-1} |x_2|^{\beta -1}}{|x_1|^{2\alpha} |x_2|^{2\beta}+x_3^2}\bigg\}$$
and proved that convolution with $\mathcal K$ is bounded on
$L^2(\Bbb R^3)$.

It is easy to see that when $\alpha=\beta=1$, $\mathcal K(x_1, x_2, x_3)$
satisfies all conditions in Corollary \ref{Cor 1.2} and Theorem \ref{Thm 1.3}.
Therefore, by Theorem \ref{Thm 1.3}, the convolution singular
integral operator $\mathcal K\ast f$ with $\alpha=\beta=1$ is also bounded on $L^p(\Bbb
R^3)$ for $1<p<\infty,$ where $\mathcal K\ast f$ is defined by the
limit of $\mathcal K_\epsilon^N\ast f$ in the $L^p$, $1<p<\infty$, norm.
It is worthwhile to point out that the theory we are developing here can be easily
generalized to the ``anisotropic" case (adapted to $\delta (s, t)$ in \eqref{eq 2.19}). The details are left to the interested reader.

\end{document}